\documentclass[a4paper]{amsart}
\usepackage{graphicx}
\usepackage[bookmarksnumbered,colorlinks,plainpages]{hyperref}
\usepackage{extarrows}
\usepackage{float}
\usepackage{amsthm}
\usepackage[all,2cell]{xy}
\usepackage{amsfonts}
\usepackage{amsmath}
\usepackage{amssymb}
\usepackage{xypic,leqno,amscd,amssymb,pstricks,latexsym,amsbsy,xypic,mathrsfs,verbatim}
\usepackage{multirow}
\usepackage{booktabs}
\usepackage{makecell}
\usepackage{cases}
\usepackage{appendix}
\usepackage{graphicx,colortbl}
\usepackage{tikz-cd}
\usepackage{etex}
\usepackage{mathtools}
\usepackage[bookmarksnumbered,colorlinks,plainpages]{hyperref}
\hypersetup{
	colorlinks=true,
	citecolor=blue
}
\newtheorem{theorem}{Theorem}[section]

\newtheorem{corollary}[theorem]{Corollary}
\newtheorem{lemma}[theorem]{Lemma}
\newtheorem{proposition}[theorem]{Proposition}
\newtheorem{definition-proposition}[theorem]{Definition-Proposition}

\newtheorem{conjecture}[theorem]{Conjecture}
\theoremstyle{definition}
\newtheorem{definition}[theorem]{Definition}

\newtheorem{remark}[theorem]{Remark}

\numberwithin{equation}{section}

\newcommand{\CC}{\mathcal{C}}
\newcommand{\OO}{{\mathcal O}}

\renewcommand{\AA}{\mathbb{A}}
\renewcommand{\L}{\mathbb{L}}
\newcommand{\X}{\mathbb{X}}
\newcommand{\Y}{\mathbb{Y}}
\renewcommand{\P}{\mathbb{P}}
\newcommand{\Z}{\mathbb{Z}}

\renewcommand{\c}{\vec c}
\newcommand{\vdelta}{\vec{\delta}}
\newcommand{\y}{\vec y}
\newcommand{\x}{\vec x}

\newcommand{\vell}{\vec {\ell}}
\newcommand{\z}{\vec z}

\newcommand{\s}{\vec s}

\newcommand{\w}{\vec{\omega}}

\newcommand{\cut}{\ar@{-}@[|(5)]}

\newcommand{\Hom}{\operatorname{Hom}\nolimits}

\newcommand{\Mod}{\operatorname{Mod}\nolimits}
\newcommand{\End}{\operatorname{End}\nolimits}
\newcommand{\Ext}{\operatorname{Ext}\nolimits}

\renewcommand{\Im}{\operatorname{Im}\nolimits}
\newcommand{\coker}{\operatorname{Cok}\nolimits}

\newcommand{\bo}{\operatorname{b}\nolimits}
\newcommand{\rk}{\operatorname{rank}\nolimits}

\newcommand{\RHom}{\mathbf{R}\strut\kern-.2em\operatorname{Hom}\nolimits}
\newcommand{\Hhom}[3]{\mathcal{H}om_{#1}(#2,#3)}

\DeclareMathOperator{\add}{\mathsf{add}}

\DeclareMathOperator{\thick}{\mathsf{thick}}
\DeclareMathOperator{\CM}{\mathsf{CM}}
\DeclareMathOperator{\ACM}{\mathsf{ACM}}
\DeclareMathOperator{\moduleCategory}{\mathsf{mod}} \renewcommand{\mod}{\moduleCategory}
\DeclareMathOperator{\coh}{\mathsf{coh}}
\newcommand{\DDD}{\mathsf{D}}
\newcommand{\KKK}{\mathsf{K}}

\DeclareMathOperator{\lb}{\mathsf{line}}
\DeclareMathOperator{\vect}{\mathsf{vect}}
\DeclareMathOperator{\proj}{\mathsf{proj}}

\tikzset{every picture/.style={line width=0.75pt}} 

\begin{document}

	\title[BP singularities via  ACM bundles on GL projective spaces]{Brieskorn-Pham singularities via ACM bundles on Geigle-Lenzing projective spaces}

	\author[J. Chen, S. Ruan and W. Weng] {Jianmin Chen, Shiquan Ruan and Weikang Weng$^*$}

	\thanks{$^*$ the corresponding author}
	\makeatletter \@namedef{subjclassname@2020}{\textup{2020} Mathematics Subject Classification} \makeatother
	
	
	\begin{abstract} 

	We study the singularity category of the Brieskorn-Pham singularity $R=k[X_1, \dots, X_4]/(\sum_{i=1}^{4} X_i^{p_i})$,  associated with the  Geigle-Lenzing projective space $\X$ of weight quadruple $(p_1,\dots, p_4)$,  by investigating the stable category $\underline{\ACM} \, \X$ of arithmetically Cohen-Macaulay bundles on $\X$.
		
	We introduce the notion of $2$-extension bundles on $\X$, which is a higher dimensional analog of extension bundles on a weighted projective line of Geigle-Lenzing, and then establish a correspondence between $2$-extension bundles and a certain important class of  Cohen-Macaulay $R$-modules studied by Herschend-Iyama-Minamoto-Oppermann. Furthermore, we construct a tilting object in $\underline{\ACM} \, \X$ consisting of $2$-extension bundles, whose endomorphism algebra is a $4$-fold tensor product of certain Nakayama algebras. We also investigate 
	the Picard group action on $2$-extension bundles and obtain an explicit formula for the orbit number, 	
	which gives a positive answer to a higher version of an open question raised by Kussin-Lenzing-Meltzer.
 	
	\end{abstract}
	
	\maketitle
	
	\section{Introduction}

	The study of singularity theory has a high contact with many mathematical subjects, including representation theory \cite{HIMO,IT,KLM}, cluster theory \cite{AIR}, algebraic geometry and mathematical physics \cite{Orlov:2004, Orlov:2009}.   
	A fundamental result due to Buchweitz \cite{Bu} states that for a  Gorenstein ring $R$,  the singularity category  of $R$ is triangle equivalent to the stable category of maximal Cohen-Macaulay $R$-modules, see also Orlov \cite{Orlov:2009} for the graded situation.

	The Brieskorn-Pham singularity of a sequence $(p_1,\dots,p_n)$ with integers $p_i\ge 2$ over an algebraically closed field $k$, is defined as  $R:=k[X_1, \dots, X_n]/(\sum_{i=1}^{n} X_i^{p_i})$. Brieskorn–Pham singularities and their (graded) singularity categories have been actively studied,  e.g.  homological mirror symmetry \cite{FU}.   In particular, for the case $n=3$, they are called  triangle singularities and their singularity categories  have been studied in various approaches, such as the stable category of vector bundles on weighted projective lines of Geigle-Lenzing \cite{KLM,KLM2,L1}, Cohen-Macaulay modules and matrix factorizations \cite{KST1,KST2}.

	Geigle-Lenzing projective spaces, introduced by Herschend, Iyama, Minamoto and Oppermann in \cite{HIMO},  are higher dimensional analogs of the weighted projective lines. The singularity category of the (graded) Brieskorn-Pham singularity (or more general, Geigle-Lenzing complete intersection)   associated with the Geigle-Lenzing projective space $\X$, was studied in terms of Cohen–Macaulay representations.	

	From a different perspective, it is also indicated in \cite{HIMO} that the
	Cohen-Macaulay modules over the  Brieskorn-Pham singularity, correspond to a particular class of vector bundles, called arithmetically Cohen-Macaulay (ACM) bundles on $\X$. With motivation coming from the work in \cite{KLM},  
	we study  the singularity category of the  Brieskorn-Pham singularity  via the structure of  the corresponding stable category of	ACM bundles on $\X$.

	Throughout this paper, we focus on  the Brieskorn-Pham singularity $$R:=k[X_1,  \dots, X_4]/(\sum_{i=1}^{4} X_i^{p_i})$$ of a quadruple $(p_1,\dots,p_4)$ with integers $p_i\ge 2$, where the base field $k$ is supposed to be algebraically closed. Let $\L$ be the abelian group on generators $\x_1,\ldots,\x_4,\c$ modulo the relations $p_1 \x_1= \dots=p_4 \x_4=:\c$. The category $\coh \X$ of coherent sheaves on the \emph{Geigle-Lenzing} (\emph{GL}) \emph{projective space} $\X$ attached to $R$, is defined by 	applying Serre's construction \cite{S1} to the $\L$-graded singularity $R$. Recall that the Picard group of $\X$ is naturally isomorphic to the abelian group $\L$. 
	
	We denote by $\pi:\mod^{\L}R\to\coh\X$ the sheafification functor. Denote by $\CM^{\L}R$ the category  of $\L$-graded (maximal) Cohen–Macaulay $R$-modules and by  $\ACM \X$ its sheafification, called the category of 
	arithmetically Cohen-Macaulay 
	bundles on $\X$, which is Frobenius,  and  the  stable category $\underline{\ACM} \, \X$ forms a triangulated category 
	by a general result of Happel \cite{Hap1}  (see Section \ref{sec: Preliminaries} for details). 
	It follows from \cite{HIMO} that the  sheafification yields a  natural equivalence $\pi: \CM^{\L}R\xrightarrow{\sim} \ACM \X$, mapping the indecomposable projective module $R(\x)$ to the line bundle $\OO(\x)$ for any $\x \in \L$, and further induces triangle equivalences   $$\underline{\ACM} \, \X \simeq \underline{\CM}^{\L} R \simeq \DDD^{\L}_{\rm sg}(R),$$ where  $\DDD^{\L}_{\rm sg}(R):=\DDD^{\bo}(\mod^{\L}R)/\KKK^{\bo}(\proj^{\L}R)$
	is the  \emph{singularity category} of $R$, and $\underline{\CM}^{\L}R$ is the stable category of ${\CM}^{\L} R$.
	
	The main purpose of the paper is to study the singularity category of $R$ via the structure of the corresponding stable category of ACM bundles on $\X$. Recall that there are two  distinguished elements of $\L$, the \emph{dualizing element} $\w:=\c-\sum_{i=1}^{4}\x_i$ and the \emph{dominant element} $\vdelta:=2\c+2\w=\sum_{i=1}^{4}(p_i-2)\x_i$.

		Let  $L$ be a line bundle and $\x=\sum_{i=1}^{4}\lambda_i\x_i$ with $0\le \x \le \vdelta$.
		There exists a `unique' nonsplit exact sequence (see Section \ref{sec:2-extension bundles and $2$-coextension bundles} for details)
		\begin{equation*}
			0 \to L(\w) \xrightarrow{}  E 
			\xrightarrow{} 
			\bigoplus_{1\le i \le 4} L(\x-(1+\lambda_i)\x_i) \xrightarrow{\gamma} L(\x) \to 0
		\end{equation*}
		in $\coh\X$, where $\gamma:=(X_i^{\lambda_i+1})_{1\le i\le 4}$. The term $E_L \langle \x \rangle:=E$ of the sequence,  is called the \emph{$2$-extension bundle}. For $\x=0$, it is called a \emph{$2$-Auslander bundle}. 
		Each $2$-extension bundle is an exceptional ACM bundle of rank four, see Propositions \ref{homological properties of $E$ and $K$ 2}, \ref{homological properties of $E$ and $K$ 3} and \ref{shift preserve extension bundle}.
	 Dually,  there exists a `unique' nonsplit exact sequence 
	 $$0 \to L(\w)  \xrightarrow{\gamma'} \bigoplus_{1\le i \le 4} L(\w+(1+\lambda_i)\x_i)  \to F \to L(\x) \to 0$$  	in $\coh\X$, where $\gamma':=(X_i^{\lambda_i+1})_{1\le i\le 4}$. The term $F_L \langle \x \rangle:=F$ of the sequence,  is called the \emph{$2$-coextension bundle}. 
	 For $\x=0$,	it is called a \emph{$2$-coAuslander bundle}. 	 		
	 There are close relationships between $2$-extension bundles and $2$-coextension bundles  shown in Propositions \ref{dual ext-def} and \ref{susp.ext.}.

	We establish a correspondence between $2$-extension bundles and a certain class of  Cohen-Macaulay modules studied in \cite{HIMO}. 
	Let $\s:=\sum_{i=1}^{4}\x_i$. For each element $\vell=\sum_{i=1}^{4}\ell_i\x_i$ in the interval $[\s,\s+\vdelta]$, let
	$$U^{\vell}:=\rho(R/(X_i^{\ell_i}\mid 1\le i\le 4)),$$ 
	where $\rho$ denotes the composition $\DDD^{\bo}(\mod^{\L}R)\to\DDD_{\rm sg}^{\L}(R)\xrightarrow{\sim} \underline{\CM}^{\L}R$. For simplicity,  we write $E \langle \x \rangle:=E_{\OO} \langle \x \rangle$ and $F \langle \x \rangle:=F_{\OO} \langle \x \rangle$.

	\begin{theorem}
		[Theorem~\ref{correspondence of 2-Extension bundles}, Corollary~\ref{bijections}] \label{Thm A}
		Let $R$ be an $\L$-graded Brieskorn-Pham singularity of a quadruple $(p_1,\dots,p_4)$, and
		$\X$ be the corresponding GL projective space. 
		Then for each $\vell \in [\s,\s+\vdelta]$, we have   
		$$\pi(U^{\vell})= E\langle \s+\vdelta-\vell \rangle(-\w).$$	
		Conversely, each $2$-extension bundle is of this form, up to degree shift.		
		Moreover, there are bijections between:
			\begin{itemize}			
				\item[(a)]   The set of the $2$-extension bundles $E \langle \x \rangle$, $0 \leq \x \leq \vdelta;$
				\item[(b)]   The set of  the $2$-coextension bundles $F \langle \x \rangle$, $0 \leq \x \leq \vdelta;$
				\item[(c)]   The set of the $\L$-graded Cohen-Macaulay modules $U^{\vell}$, $ \vell \in [\s, \s+\vdelta]$.
			\end{itemize}				
	\end{theorem}

		We construct a tilting object in $\underline{\ACM} \, \X$ consisting of $2$-(co)Auslander bundles. For each $0\le \x \le \vdelta$ with $\x=\sum_{i=1}^{4} \lambda_i\x_i$ in  normal form, we let $ \sigma(\x):= \sum_{i=1}^{4} \lambda_i$. Denote by $k\vec{\AA}_n$  the path algebra of the equioriented quiver of type $\AA_n$ and  by $\vec{\AA}_{n}(2):=k\vec{\AA}_n/{\rm rad^2} \, k\vec{\AA}_n$, where $n \ge 1$. 
		Denote by $[1]$  the suspension functor in the triangulated category $\underline{\ACM} \, \X$.

		\begin{theorem}[Theorem~\ref{main  theorem}] \label{Thm C}		
		Let $\X$ be a GL projective space attached to an $\L$-graded Brieskorn-Pham singularity of a quadruple $(p_1,\dots,p_4)$. Let  $E_L$ be a $2$-Auslander bundle for some line bundle $L$. 	 		
		Then  $$T=\bigoplus_{0\leq\x\leq\vdelta}E_L(\x)[-\sigma(\x)]$$
		is a tilting object in $\underline{\ACM}\, \X$ with endomorphism algebra $$\underline{\End}(T)^{\rm op}\simeq \bigotimes_{1\le i \le 4} k\vec{\AA}_{p_i-1}(2). $$
	\end{theorem}

	Denote by $\mathcal{V}$ the set of isomorphism classes of $2$-extension bundles in $\ACM \X$. The Picard group $\L$ acts on $\mathcal{V}$ by line bundle twist
	$$\L\times\mathcal{V} \to \mathcal{V}, \  \ \ (\x, E) \mapsto E(\x).$$
	We describe when $2$-extension bundles are in the same $\L$-orbit in Proposition \ref{iso. exten. bundle}.
	Basing on this observation, we get the formula for the cardinality of orbits of $\mathcal{V}/\L$ via the group $\Z_2 \times\Z_2\times\Z_2$ action on certain set of $2$-extension bundles,  giving a positive answer to a higher version of the open question stated in \cite[Remark 9.4]{KLM}.

	\begin{theorem}
		[Theorem~\ref{exact number}] \label{Thm D}
		Let  $J=\{1\le i \le 4 \mid p_i \ \text{is even}\}$. Then we have 
		$$|\mathcal{V}/\L|=\frac{1}{8}\sum_{\substack{I \subset J \\ |I|  \text{ even}}}\prod_{  i\in  I^{c}} (p_i-1),$$
		where $I^c:=\{1,\dots,4\}\setminus I$.		
		\end{theorem}

	The paper is organized as follows.  In Section \ref{sec: Preliminaries}, we recall some basic concepts and properties on   Brieskorn-Pham singularities and  Geigle-Lenzing projective spaces. 
	In Section \ref{sec:2-extension bundles and $2$-coextension bundles}, we introduce the notions of $2$-extension bundles and $2$-coextension bundles,   and their properties are studied in Section \ref{sec:Basic properties of 2-extension bundles}.   	
	By establishing the correspondence between $2$-extension bundles and the class of Cohen-Macaulay modules, we prove Theorem \ref{Thm A} in Section \ref{sec:The correspondence of 2-Extension bundles}. 
	In Section \ref{sec:Tilting theory in the stable category of ACM bundles}, 	we study tilting objects in the stable category of arithmetically Cohen-Macaulay bundles and prove Theorem \ref{Thm C}.  
	In Section \ref{sec:The action of 2-extension bundles}, we investigate the action of the Picard group 	on $2$-extension bundles and prove Theorem \ref{Thm D}.

	\section{Preliminaries}
	\label{sec: Preliminaries}
	In this section, we recall some basic notions and facts appearing in this paper. We refer to  \cite{HIMO,KLM} on Geigle-Lenzing projective spaces. We restrict our treatment to the case of that of dimension $d=2$, that is, Geigle-Lenzing  projective planes, and where the base field $k$ is supposed to be algebraically closed.   
	
	\subsection{Quadrangle singularities}	
	 Fix a quadruple $(p_1,\dots, p_4)$ with integers $p_i\ge 2$ called \emph{weights}. We consider the $k$-algebra
	$$R:=k[X_1, \dots, X_4]/(\sum_{i=1}^{4}X_i^{p_i}).$$ Let $\L$ be the abelian group on generators $\x_1,\ldots,\x_4,\c$ subject to the relations $$p_1 \x_1= \dots=p_4 \x_4=:\c.$$ The element $\c$ is called the \emph{canonical element} of $\L$. Each element $\x$ in $\L$ can be uniquely written in a \emph{normal form} as 
	\begin{align}\label{equ:nor}
		\x=\sum_{i=1}^4\lambda_i\x_i+\lambda\c
	\end{align}
	with $ 0\le \lambda_i < p_i$ and $\lambda\in\Z$. We can regard $R$ as an $\L$-graded $k$-algebra by setting $\deg X_i:=\x_i$ for any $i$, hence $R=\bigoplus_{\x\in \L } R_{\x}$, where $R_{\x}$ denotes the homogeneous component of degree $\vec{x}$.  We call $R$ an $\L$-graded \emph{quadrangle singularity}, which is a special case of the graded Brieskorn-Pham singularity in \cite{FU,HM}.
	We also call the pair $(R,\L)$ Geigle-Lenzing hypersurface of dimension three in \cite{HIMO}. 
	
	Let $\L_+$ be the submonoid of $\L$ generated by $\x_1,\ldots,\x_4,\c$. Then we equip $\L$ with the structure of a partially ordered set: $\x\le \y$ if and only if $\y-\x\in\L_+$. Each element $\x$ of $\L$ satisfies exactly one of the following two possibilities
	\begin{align}\label{two possibilities of x}
		\x\geq 0 \text{\quad or\quad}  \x\leq 2\c+\vec{\omega},
	\end{align}
	where  $\w=\c-\sum_{i=1}^4 \x_i$ is called the \emph{dualizing element} of $\L$. We have $R_{\x} \neq 0$ if and only if $\x \geq 0$ if and only if $\ell \geq 0$ in its normal form in (\ref{equ:nor}).

	Denote by $\CM^{\L} R$ the category of $\L$-graded (maximal) Cohen-Macaulay (CM) $R$-modules, which is Frobenius, and by $\underline{\CM}^{\L}R$ its stable category. By results of Buchweitz \cite{Bu} and Orlov \cite{Orlov:2009}, there exists a triangle equivalence $\DDD^{\L}_{\rm sg}(R)\simeq\underline{\CM}^{\L}R$, where 
	$\DDD^{\L}_{\rm sg}(R):=\DDD^{\bo}(\mod^{\L}R)/\KKK^{\bo}(\proj^{\L}R)$
	is the \emph{singularity category} of $R$. 	
	The following presents a basic property in Cohen-Macaulay representation theory. Here, we denote by $D$ the $k$-dual, that is  $D(-):=\Hom_{k}(-,k)$.
	\begin{theorem}[\cite{HIMO}] (Auslander-Reiten-Serre duality) There exists a functorial isomorphism for any $X,Y\in \underline{\CM}^{\L}R$:
	\begin{equation}\label{Auslander-Reiten-Serre duality CM}
		\Hom_{\underline{\CM}^{\L}R}(X,Y)\simeq D\Hom_{\underline{\CM}^{\L}R}(Y,X(\w)[2]).
	\end{equation}	
	\end{theorem}

	\subsection{Geigle-Lenzing projective spaces}	
	The category of \emph{coherent sheaves} on
	\emph{Geigle-Lenzing} (\emph{GL}) \emph{projective space} $\X$ is defined as the quotient category	
	\[ \coh \X :=\mod^{\L}R/\mod_0^{\L}R\]
	of finitely generated $\L$-graded $R$-modules $\mod^{\L}R$ by its Serre subcategory $\mod^{\L}_0R$ consisting of finite dimensional $R$-modules. 
	
	Denote by $\pi:\mod^{\L}R\to\coh\X$ the natural functor, also called \emph{sheafification}. The object $\OO:=\pi(R)$ is called the \emph{structure sheaf} of $\X$. We recall a list of fundamental properties of  $\coh\X$ and its derived category $\DDD^{\bo}(\coh\X)$.
	
	\begin{theorem}[\cite{HIMO}] 
		\label{basic properties} The category $\coh\X$ has the following properties:
		\begin{itemize}
			\item[(a)] $\coh\X$ is a Noetherian abelian category of global dimension $2$.
			\item[(b)] $\Hom_{\DDD^{\bo}(\coh\X)}(X,Y)$ is
			a finite dimensional $k$-vector space for all $X,Y\in\DDD^{\bo}(\coh\X)$. In particular, $\coh\X$ is Ext-finite.
			\item[(c)] For any $\x,\y \in \L$ and $i\in\Z$,  
			\begin{equation}\label{extension spaces between line bundles}
				\Ext_{\X}^i(\OO(\x),\OO(\y))=\left\{
				\begin{array}{ll}
					R_{\y-\x}&\mbox{if $i=0$,}\\
					D(R_{\x-\y+\w})&\mbox{if $i=2$,}\\
					0&\mbox{otherwise.}
				\end{array}\right.
			\end{equation}
			\item[(d)] (Auslander-Reiten-Serre duality) There exists a functorial isomorphism for any $X, Y\in\DDD^{\bo}(\coh\X) \mathsf{:}$
			\begin{equation}\label{Auslander-Reiten-Serre duality}
				\Hom_{\DDD^{\bo}(\coh\X)}(X,Y)\simeq D\Hom_{\DDD^{\bo}(\coh\X)}(Y,X(\w)[2]).
			\end{equation}
			In other words, $\DDD^{\bo}(\coh\X)$ has a Serre functor $(\w)[2]$.	
		\end{itemize}
	\end{theorem}
	
	Denote by $\vect \X$ the full subcategory of $\coh \X$ formed by all vector bundles 
	 and by $\lb\X$ its full subcategory  of finite direct sums of 
	 line bundles $\OO(\x)$ with $\x\in \L$. 
	The following lemma gives an equivalent description of vector bundles.

	\begin{lemma}[\cite{HIMO}] \label{second syzygy}
			Any object in $\vect\X$ is isomorphic to $\pi(X)$ for some
			$X\in\mod^{\L}R$ such that there exists an exact sequence
			$0\to X\to P^0\to P^1$ in $\mod^{\L}R$ with $P^0,P^1\in\proj^{\L}R$.		
	\end{lemma}

	The natural functor  $\pi:\mod^{\L}R\to\coh\X$ restricts to a fully faithful 
	functor $\CM^{\L}R\to\vect\X$ and further induces an equivalence $\CM^{\L}R \xrightarrow{\sim} \pi(\CM^{\L} R)$.  In the context of projective geometry (e.g. \cite{CH,CMP}), the objects in $\pi(\CM^{\L} R)$ are called \emph{arithmetically Cohen-Macaulay} (\emph{ACM}) \emph{bundles}. We denote by $\ACM \X:=\pi(\CM^{\L} R)$. The following gives a description of $\ACM \X$ in terms of vector bundles by \cite{HIMO}.
		\begin{align}\label{ACM bundle eq.}
				\ACM \X&=
			\{X\in\vect\X\mid \Ext_{\X}^1(\OO(\x),X)=0 \text{ for any} \ \x\in\L\}\\
				&=\{X\in\vect\X\mid \Ext_{\X}^1(X,\OO(\x))=0\text{ for any} \ \x\in\L\}\nonumber. 
			\end{align}			
	Notice that all line bundles serve as the indecomposable projective-injective objects in $\ACM \X$ by (\ref{ACM bundle eq.}). Moreover, $\pi$ restricts to an equivalence $\proj^{\L}R\xrightarrow{\sim}\lb\X$ with the $\L$-graded  indecomposable projective module $R(\x)$ corresponding to the line bundle $\OO(\x)$. Therefore, this turns $\ACM \X$ into a Frobenius category and thus its stable category $\underline{\ACM} \, \X$ is triangle equivalent to $\underline{\CM}^{\L}R$.

	\subsection{Tilting bundles and Grothendieck groups} 
	Let $\mathcal{T}$ be a triangulated category with a suspension functor $[1]$.
	
	\begin{definition} We say that an object $T \in \mathcal{T}$ is \emph{tilting} if
		\begin{itemize}
			\item[(a)] $T$ is \emph{rigid}, that is $\Hom_{\mathcal{T}}(T,T[i])=0$ for all $i \neq 0$.
			\item[(b)] $\thick T=\mathcal{T}$, where $\thick T$ denotes by the smallest thick subcategory of $\mathcal{T}$ containing $T$.		
		\end{itemize}
	\end{definition}

		Let $K_0(\coh \X)$ be the Grothendieck group of $\coh \X$. We denote by $[X]$ the corresponding element in $K_0(\coh \X)$. It is shown in \cite{HIMO} that $K_0(\coh \X)$ is a free abelian group  with a  basis $[\OO(\x)]$ for $0 \le \x \le 2\c$. The \emph{rank functor} rank: $K_0(\coh \X) \to \Z$ is the additive function uniquely determined by $${\rm rank}\,([\OO(\x)])=1\ \ \text{for any}\ \ \x\in \L.$$ 
		Recall that there are no nonzero morphisms from torsion sheaves (that is, rank zero objects in $\coh \X$) 
	to vector bundles by a general result of \cite[Proposition 3.28]{HIMO}.

	\begin{proposition} \label{k0}
		 Let $\X$ be a GL projective space with  weight quadruple.
		Assume that $E$ and $F$ be two exceptional objects in $\vect \X$. If $[E]=[F]$, then   $E\simeq F$.		
	\end{proposition}
	\begin{proof} Notice that for any $X \in \vect \X$, $[X]=0$ implies that  $\rk X=0$ and thus $X=0$ by a general result of \cite[Proposition 3.28]{HIMO}.
		Hence the argument in \cite[Proposition 4.4.1]{Me} still works in the vector bundle on GL projective space.	
	\end{proof}

	\subsection{Yoneda $2$-extension groups} Let $\mathcal{A}$ be an abelian category. For two objects $X,Y\in\mathcal{A}$,  let $\Ext_{\mathcal{A}}^{2}(X,Y)$ denote the abelian group of \emph{$2$-extensions} in the sense of Yoneda, that is, equivalence classes of exact sequences	
	$$\xi: \quad  0\to Y \to B \to A \to X \to 0,$$
	where such sequences $\xi$ and $\xi'$ are \emph{equivalent} if there exists an exact sequence $\xi''$ fitting into the following commutative diagram.
	 	\[
		\begin{tikzcd} 
		\xi:\,\,\, \quad	0\rar&	Y \rar\ar[equals]{d}& B \dar \rar& A \dar \rar& X\ar[equals]{d} \rar &0 \\
		\xi'': \quad	0\rar&	Y\ar[equals]{d}\rar&B'' \rar & A'' \rar & X \ar[equals]{d} \rar&0 \\
		\xi':\, \quad	0\rar&	Y\rar&B' \rar\ar{u} & A'\ar{u} \rar & X  \rar &0 
		\end{tikzcd}
		\]
	Note that $\xi$ is split if and only if $\xi$ is equivalent to $0\to Y \xrightarrow{1_{Y}} Y\xrightarrow{0}  X \xrightarrow{1_{X}} X \to 0$. 

	\begin{lemma} \label{2-Yoneda}Let $\xi: 0\to Y \xrightarrow{f} B \xrightarrow{g} A \xrightarrow{h} X \to 0$ be an exact sequence in $\mathcal{A}$, and $B\xrightarrow{p}K \xrightarrow{i} A$ be the canonical factorization of $g$. Then the sequence $\xi$ is split if and only if there exists an object $Z$ in $\mathcal{A}$ fitting into the commutative diagram 	
		\[
		\begin{tikzcd} 
			& 0 \dar& 0\dar & 0\dar &\\
			0\rar &	Y\rar{f}\ar[equals]{d}&B\rar{p}\dar{}&K \dar{i}\rar & 0\\
			0\rar &	Y\rar{f'}&Z\rar{p'}\dar{}&A  \dar{h}\rar & 0 \\
			& & X\ar[equals]{r} \dar & X\dar& \\
			&  & 0 & 0 & 
		\end{tikzcd}
		\]	
		with exact rows and columns in $\mathcal{A}$. 	
	\end{lemma}
	\begin{proof} Assume that $\xi$ is split, that is, $\Ext_{\mathcal{A}}^{2}(X,Y)=0$. 	
	By applying the functor $\Hom_{\mathcal{A}}(-,Y)$  to the exact sequence $0 \to K \xrightarrow{i} A \xrightarrow{h} X \to 0$, yields an epimorphism $\Ext_{\mathcal{A}}^{1}(A,Y)\to \Ext_{\mathcal{A}}^{1}(K,Y)$. Thus we have the commutative diagram above.
		
		On the other hand, we get the morphisms from both exact sequences $\xi$ and $\xi':0\to Y \xrightarrow{1_{Y}} Y\xrightarrow{0}  X \xrightarrow{1_{X}} X \to 0$ to the exact sequence  $$\xi'': 0 \to Y \xrightarrow{f'} Z  \xrightarrow{\left(\begin{smallmatrix} p' \\  0  \end{smallmatrix}\right)}
		A\oplus X \xrightarrow{(0~~1_{X})} X \to 0.$$ 
		Hence we have that $\xi$ is split.
	\end{proof}
	
	For convenient, we denote ${\Ext}^{i}_{\coh \X}(-,-)$ by ${\Ext}^{i}(-,-)$, even more simplified by ${}^{i}(-,-)$ for $i\ge 0$, and denote ${\Hom}_{\underline{\ACM}\, \X}(-,-)$ by  $\underline{\Hom}(-,-)$.
	
	\section{$2$-extension bundles and $2$-coextension bundles}
	\label{sec:2-extension bundles and $2$-coextension bundles}
	In this section, we introduce  the notions of $2$-extension bundles and $2$-coextension bundles on GL projective planes, which are higher dimensional analogs of extension bundles on weighted projective lines.
	Recall that  $$\vdelta :=2\c+2\w=\sum_{i=1}^{4}(p_i-2)\x_i$$ is called the $\emph{dominant element}$ of $\L$. Observe that $0 \leq \x \leq \vdelta$ holds if and only if we have $\x=\sum_{i=1}^4\lambda_i\x_i$ with $0 \le \lambda_i \le p_i-2$ for  $1\le i\le 4$.

		\subsection{$2$-extension bundles} 
	For any $\x=\sum_{i=1}^4\lambda_i\x_i$ with $0\le \x \le \vdelta$, we consider the morphism 
	\begin{equation*}
		\bigoplus_{1\le i \le 4}R(\x-(1+\lambda_i)\x_i) \xrightarrow{\gamma} R(\x)
	\end{equation*}
	in $\mod^{\L} R$, where $\gamma:=(X_i^{\lambda_i+1})_{1\le i\le 4}$. Since $\coker \gamma= (R/(X_i^{\lambda_i+1}\mid 1\le i\le 4)(\x)$ belongs to $\mod_0^{\L}R$, 
	by degree shift to its image in $\coh \X$, this yields an epimorphism 
	\begin{equation}
		\bigoplus_{1\le i \le 4}L(\x-(1+\lambda_i)\x_i) \xrightarrow{\gamma} L(\x),
	\end{equation}
	 where $L$ is a line bundle. For simplicity, we write 
	 \begin{align}
	 	 L_i\langle \x \rangle:= L(\x-(1+\lambda_i)\x_i) \ \ \text{and}\ \ L\langle \x \rangle:=\bigoplus_{1\le i\le4}L_i\langle \x \rangle.
	 \end{align}	
	Let $\eta_{2}: 0 \to K_{\x,L} \xrightarrow{i} L\langle \x \rangle \xrightarrow{\gamma} L(\x) \to 0$ be an exact sequence
	 in $\coh \X$ extending the morphism $\gamma$.  
	 The subscript of $K_{\x,L}$ is usually omitted below. By applying $\Hom(-,L(\w))$ to the  sequence $\eta_{2}$, we obtain an exact sequence
	$$0={}^{1}(L\langle \x \rangle,L(\w)) \to {}^{1}(K,L(\w)) \to {}^{2}(L(\x),L(\w))\to {}^{2}(L\langle \x \rangle,L(\w))=0.$$
	Hence $\Ext^{1}(K,L(\w))\simeq \Ext^{2}(L(\x),L(\w))\simeq D\Hom(L,L(\x))=k$. 
	 Let $\eta_{1}:0\to L(\w) \xrightarrow{\alpha} E \xrightarrow{p} K \to 0$ be a nonsplit exact sequence.  By connecting these two sequences $\eta_{1}$ and $\eta_{2}$, we finally obtain an exact sequence 
	 \begin{equation}\label{2-extension bundle exact}
	 \eta: \quad	0 \to L(\w) \xrightarrow{\alpha}  E 
	 	\xrightarrow{\beta} 
	 	L\langle \x \rangle \xrightarrow{\gamma} L(\x) \to 0
	 \end{equation}
	  in $\coh \X$. In this case,
 the term $E_L \langle \x \rangle:=E$ of the sequence, which is uniquely defined up to isomorphism, is called the \emph{$2$-extension bundle} gived by the data $(L,\x)$. If $L=\OO$, then we just write $E\langle \x \rangle$. For $\x=0$, the term $E_L:=E_L\langle 0 \rangle$ is called the \emph{$2$-Auslander bundle} associated with $L$.

	 In the setup above, we start with proving the following results, 
	 which collect some homological properties of $E$ and $K$.
	 \begin{proposition}\label{homological properties of $E$ and $K$ 1} Assume $0\le \y\le \x \le \vdelta$. Then the following assertions hold.
	 \begin{itemize} 
	 	\item[(a)] The sequence {\rm{(\ref{2-extension bundle exact})}} does not split.
	 	\item[(b)]   $\Ext^{i}(L(\y),K)=0$ for any $i \neq 1$, and $\Ext^{1}(L(\y),K)=k$.
	 	\item[(c)]  $\Ext^{i}(L(\y),E)=0$ for any $i\ge 0$.
	 \end{itemize}
	 \end{proposition}
	 \begin{proof}
	 	 	 (a) Assume that the sequence {\rm{(\ref{2-extension bundle exact})}} is split.
	 		By Lemma \ref{2-Yoneda}, there exists an object $X$ in $\coh \X$ fitting into the following commutative diagram. 
	 		\[
	 		\begin{tikzcd} 
	 			& 0 \dar& 0\dar & 0\dar &\\
	 			0\rar &	L(\w)\rar{\alpha}\ar[equals]{d}&E\rar{p}\dar&K \dar{i}\rar & 0\\
	 			0\rar &	L(\w)\rar&X\rar\dar&L \langle \x \rangle \dar{\gamma}\rar & 0 \\
	 			& & L(\x)\ar[equals]{r} \dar & L(\x)\dar& \\
	 			&  & 0 & 0 & 
	 		\end{tikzcd}
	 		\]	
	 		Since $\Ext^{1}(L \langle \x \rangle,L(\w))=0$,  the middle row of the diagram is split, implying the top row of diagram is also split, a contradiction to our assumption on $\eta_1$.

	 		(b) Recall that $\coh \X $ has global dimension $2$.  We can restrict our treatment to the case for $i=0,1,2$.	 
	 		We apply $\Hom(L(\y),-)$ to  $\eta_{2}$ and get an exact sequence 
	 		\begin{align*}
	 			0&\to(L(\y),K) \to (L(\y),L\langle \x \rangle) \to (L(\y),L(\x)) \\ & \to {}^{1}(L(\y),K)   
	 			\to {}^{1}(L(\y),L\langle \x \rangle) \to {}^{1}(L(\y),L(\x)) \\ & \to{}^{2}(L(\y),K) \to{}^{2}(L(\y),L\langle \x \rangle) \to{}^{2}(L(\y),L(\x))= 0 \nonumber
	 		\end{align*}
	 		It is straightforward to check that $\Hom(L(\y),L\langle \x \rangle)=0$ and $\Ext^{2}(L(\y),L\langle \x \rangle)=0$ hold by (\ref{Auslander-Reiten-Serre duality}). Then we have $\Ext^{i}(L(\y),K)=0$ for $i=0,2$ and $\Ext^{1}(L(\y),K)\simeq \Hom(L(\y),L(\x))=k$. Thus we have the assertion.

	 		(c) 
	 		We apply $\Hom(-,L(\y+\w))$ to  $\eta_{1}$ and obtain an exact sequence 
	 		\begin{align*}
	 			0&\to(K,L(\y+\w)) \to (E,L(\y+\w)) \to (L(\w),L(\y+\w)) \\ & \xrightarrow{\partial} {}^{1}(K,L(\y+\w))
	 			\to {}^{1}(E,L(\y+\w)) \to{}^{1}(L(\w),L(\y+\w)) \\ & \to {}^{2}(K,L(\y+\w)) \to {}^{2}(E,L(\y+\w))\to{}^{2}(L(\w),L(\y+\w))= 0. \nonumber
	 		\end{align*}
	 		Combining (b) and (\ref{Auslander-Reiten-Serre duality}), we have $\Hom(K,L(\y+\w))=0$ and $\Ext^{2}(K,L(\y+\w))=0$.
	 		Now we claim that  the connecting morphism $\partial$ is nonzero, that is, the pushout of $\eta_1$ along the morphism $y:=\prod_{i=1}^4X_i^{e_i}:L(\w) \to L(\y+\w)$ is nonzero, where $\y=\sum_{i=1}^{4}e_i\x_i$ with $e_i \ge 0$. This is shown invoking Auslander-Reiten-Serre duality, yielding the following commutative diagram.
	 		\[
	 		\begin{tikzcd} 
	 			\Ext^{2}(L(\x),L(\w)) \rar{\simeq} \dar{\Ext^{2}(L(\x),y)}&  D\Hom(L(\w),L(\x+\w))\dar{D\Hom(y,L(\x+\w))}\\
	 			\Ext^{2}(L(\x),L(\y+\w))\rar{\simeq}  &D\Hom(L(\y+\w),L(\x+\w))
	 		\end{tikzcd}
	 		\]
	 		Since $0\le \y\le \x \le  \vdelta$ by our assumption, the morphism $D\Hom(y,L(\x+\w))$ is an isomorphism. Hence also $\Ext^{2}(L(\x),y)$ is an isomorphism. Since the top row of the following commutative diagram is nonsplit by  (a), this implies that the bottom row  is also nonsplit. 
	 		\[
	 		\begin{tikzcd}[column sep=1.3em, row sep=0.50em]
	 			0\ar{rr}&&L(\w)\ar[rr, "\alpha"]\ar[dd, "y"]&&E\ar[rr, "\beta"]\ar{dr}\ar{dd}&&L\langle \x \rangle\ar[rr, "\gamma"]\ar[equals]{dd}&&L(\x)\ar{rr}\ar[equals]{dd}&&0\\
	 			&&&&&K\ar{ur}\\
	 			0\ar{rr}&&L(\y+\w)\ar{rr}&&M\ar{rr}\ar{dr}&&L\langle \x \rangle\ar[rr, "\gamma"]&&L(\x)\ar{rr}&&0\\
	 			&&&&&K\ar{ur}\ar[crossing over, equals]{uu}
	 		\end{tikzcd}
	 		\]
	 		Thus the pushout of $\eta_1$ along $y$ is nonzero, as claimed. Because both the spaces $\Hom(L(\w),L(\y+\w))$ and $\Ext^{1}(K,L(\y+\w))$ are one dimensional, the nonzero morphism $\partial$ is an isomorphism. Thus
	 		we have $\Ext^{i}(E,L(\y+\w))=0$ for any $i\ge 0$.  The assertion follows immediately by using Auslander-Reiten Serre duality.
	 \end{proof}
	 
	 Recall that for an abelian or a triangulated category $\mathcal{C}$,  an object $E$ is called \emph{exceptional} in $\mathcal{C}$ if $\End_{\mathcal{C}}(E) = k$ and $\Ext_{\mathcal{C}}^{i}(E,E)=0$ hold for any $i\ge 1$. Moreover, a sequence of exceptional objects $(E_1,\dots,E_n)$ in $\mathcal{C}$ is called an  \emph{exceptional sequence} if $\Ext_{\mathcal{C}}^{p}(E_i,E_j)=0$ holds for any integers  $i>j$ and $p\in \Z$.

	  \begin{proposition}\label{homological properties of $E$ and $K$ 2}
	  	Both $E$ and $K$ are exceptional in $\coh \X$ and in $\DDD^{\bo}(\coh\X)$. 		  		  	
	  \end{proposition}
	 \begin{proof}
	 		   First we show that $K$ is exceptional in $\coh \X$.  
	 	For each $1\le j \le 4$, we apply $\Hom(L_j\langle \x \rangle,-)$ to  $\eta_{2}$ and obtain an exact sequence
	 	\begin{align*}
	 		0&\to(L_j\langle \x \rangle,K) \to (L_j\langle \x \rangle,L\langle \x \rangle) \xrightarrow{\gamma^{\ast} } (L_j\langle \x \rangle,L(\x))  \\ &\to {}^{1}(L_j\langle \x \rangle,K)
	 		\to {}^{1}(L_j\langle \x \rangle,L\langle \x \rangle) \to{}^{1}(L_j\langle \x \rangle,L(\x))  \nonumber\\  &\to {}^{2}(L_j\langle \x \rangle,K) \to {}^{2}(L_j\langle \x \rangle,L\langle \x \rangle)= 0. \nonumber
	 	\end{align*}	
	 	Because both  $\Hom(L_j\langle \x \rangle,L\langle \x \rangle)$ and $\Hom(L_j\langle \x \rangle,L(\x))$ are one dimensional spaces and the morphism $\gamma^{\ast}$ 
	 	is nonzero, this implies that $\gamma^{\ast}$ is an isomorphism.  
	 	Hence the space   $\Ext^{i}(L_j\langle \x \rangle,K)$ vanishes for any $i \ge 0$ and $1\le j \le 4$. Moreover, we have  the space  $\Ext^{i}(L\langle \x \rangle,K)=0$ for any $i \ge 0$.         
	 	
	 	Applying $\Hom(-,K)$ to the sequence $\eta_{2}$, we have an exact sequence
	 	\begin{align*}
	 		0 &= (L(\x),K) \to (L\langle \x \rangle,K) \to (K,K) \\ 
	 		&\to {}^{1}(L(\x),K)  \to{}^{1}(L\langle \x \rangle,K)  \to {}^{1}(K,K)  \\ 
	 		&\to{}^{2}(L(\x),K) \to{}^{2}(L\langle \x \rangle,K) \to{}^{2}(K,K) \to 0. \nonumber
	 	\end{align*}
	 	By Proposition \ref{homological properties of $E$ and $K$ 1}, we have $\End(K)=k$ and $\Ext^{i}(K,K)=0$ for any $i\neq 0$.
	 	
	 	Next we show that $E$ is exceptional in $\coh \X$.  
	 	By applying $\Hom(L(\w),-)$ to $\eta_2$, it is easy to check that $\Ext^{i}(L(\w),K)=0$ holds for any $i\ge 0$.
	 	Applying $\Hom(K,-)$ to $\eta_1$, we obtain an exact sequence
	 	\begin{align*}
	 		0 &\to (K,L(\w)) \to (K,E) \to (K,K) \\ &\xrightarrow{\partial}   {}^{1}(K,L(\w))  \to
	 		{}^{1}(K,E)  \to {}^{1}(K,K) \\ & \to{}^{2}(K,L(\w)) 
	 		\to{}^{2}(K,E) \to{}^{2}(K,K) \to 0. \nonumber
	 	\end{align*}
	 	By Proposition \ref{homological properties of $E$ and $K$ 1},  $\Ext^{1}(K,L(\w))\simeq D\Ext^{1}(L,K)=k$ and $\Ext^{i}(K,L(\w))=0$ hold for $i=0,2$. Since  $\eta_1$ does not split, the connecting morphism $\partial$ between one dimensional spaces is an isomorphism. Thus we have $\Ext^{i}(K,E)=0$ for any $i\ge 0$. Applying $\Hom(L(\w),-)$ to $\eta_1$, we obtain an exact sequence 
	 	\begin{align*}
	 		0 &\to (L(\w),L(\w)) \to (L(\w),E) \to (L(\w),K) \\ & \to {}^{1}(L(\w),L(\w))   \to
	 		{}^{1}(L(\w),E)  \to {}^{1}(L(\w),K) \\ &\to {}^{2}(L(\w),L(\w)) \to {}^{2}(L(\w),E) \to {}^{2}(L(\w),K) \to 0
	 	\end{align*}
	 	Thus we have $\Hom(L(\w),E)=k$ and $\Ext^{i}(L(\w),E)=0$ for any $i\neq 1$.
	 	
	 	Applying $\Hom(-,E)$ to the sequence $\eta_1$, we obtain an exact sequence
	 	\begin{align*}
	 		0 &\to (K,E) \to (E,E) \to (L(\w),E) \\ & \to {}^{1}(K,E)  \to
	 		{}^{1}(E,E)  \to {}^{1}(L(\w),E) \\ & \to{}^{2}(K,E) 
	 		\to{}^{2}(E,E) \to{}^{2}(L(\w),E) \to 0. \nonumber
	 	\end{align*}
	 	Putting things together, we have $\End(E)=k$ and $\Ext^{i}(E,E)=0$ for any $i\neq 0$. 
	 \end{proof}
	 
	 \begin{proposition}\label{homological properties of $E$ and $K$ 3} 
	 	The object $K$ belongs to $\vect \X$ and  $E$  belongs to $\ACM  \X$.	 	
	 \end{proposition}
	 \begin{proof}By Lemma \ref{second syzygy}, $K$ belongs to $\vect \X$. Because $\vect \X$ is an extension-closed subcategory of $\coh \X$, $E$ belongs to $\vect \X$. Now we  show that  $E$ satisfies the condition (\ref{ACM bundle eq.}) for being
	 	an ACM bundle. It suffices to show that $\Ext^{1}(L(\z),E)=0$ holds for any $\z \in \L$. 
	 	Applying $\Hom(L(\z),-)$ to $\eta_1$, we obtain an exact sequence
	 	\begin{align}\label{k2}
	 		0={}^{1}(L(\z),L(\w)) \to {}^{1}(L(\z),E) \to {}^{1}(L(\z),K)
	 		 \to {}^{2}(L(\z),L(\w)) 
	 	\end{align}
	 	Next we apply $\Hom(L(\z),-)$ to $\eta_2$, we obtain an exact sequence
	 	\begin{align}\label{k3}
	 		0\to (L(\z),K) \to  (L(\z),L \langle \x \rangle) \xrightarrow{\gamma^{\ast}} (L(\z),L(\x)) \to{}^{1}(L(\z),K) \to 0.
	 	\end{align}	 	
	 	Concerning the value of $\z$, we divide the proof into the following two cases.
	 	
	 	\emph{Case $1$}: $\x\not\ge \z$. In this case, we have $\Hom(L(\z),L(\x))=0$, which implies  ${\Ext}^{1}(L(\z),K)=0$. Thus we have ${\Ext}^{1}(L(\z),E)=0$.
	 		 
	 \emph{Case $2$}: $\x\ge \z$. In this case, if $\z \ge 0$, then by Proposition \ref{homological properties of $E$ and $K$ 1} we have ${\Ext}^{1}(L(\z),E)=0$; if  $\z \not\ge 0$, we have $\z \le 2\c+\w$ by (\ref{two possibilities of x}). Write $\x-\z$ in the normal form $\x-\z=\sum_{i=1}^{4}\ell_i\x_i+\ell\c$ with $0 \le \ell_i < p_i$ and $\ell \ge 0$.  We first claim that there exists $i$ with $1\le i \le 4$ such that $\x-\z \ge (1+\lambda_i)\x_i$.  The claim is clear for $\ell>0$. If  $\ell=0$, then $$\x-\z=\sum_{i=1}^{4}\ell_i\x_i \ge \x-(2\c+\w) = \sum_{i=1}^{4}(\lambda_i+1)\x_i-3\c$$ 
	 implies $3\c\ge \sum_{i=1}^{4}(\lambda_i+1-\ell_i)\x_i$. Then  
	 $\ell_i\ge \lambda_i+1$ for at least one of $1\le i \le 4$ and thus  the claim follows. 
	 To prove ${\Ext}^{1}(L(\z),E)=0$,  it suffices to show that $\gamma^{\ast}$ is an epimorphism. 
	If there exists $\ell_i\ge \lambda_i+1$ for some $1\le i \le 4$, then the multiplication morphism
	$$ X_i^{\lambda_i+1}: \ (\prod_{j\neq i} X_j^{\ell_j}) X_i^{\ell_i-\lambda_i-1} R_{\ell\c}  \to (\prod_{1\le j \le 4} X_j^{\ell_j}) R_{\ell\c} $$
	 is an isomorphism, yielding the morphism $\gamma^{\ast}$ is an epimorphism by (\ref{extension spaces between line bundles}). Otherwise $\ell_i \le \lambda_i$ for any $i$, we have
	 $\ell>0$ by the above claim and thus $\x-(1+\lambda_i)\x_i \ge \z$ for  $1\le i \le 4$. 
	  This yields that the morphism
	 $$\gamma^{\ast}=(X_i^{\lambda_i+1})_{1\le i\le 4}:\ \bigoplus_{1\le i \le 4} \Big((\prod_{j\neq i} X_j^{\ell_j})X_{i}^{p_i+\ell_i-\lambda_i-1} R_{(\ell-1)\c}\Big) \to (\prod_{1\le j \le 4} X_j^{\ell_j}) R_{\ell\c}$$ is an epimorphism. Hence the assertion follows.	 	
	 \end{proof}
	 
	 \begin{proposition}\label{homological properties of $E$ and $K$ 4} The following assertions hold.
	 	\begin{itemize}
	 		\item[(a)]	The morphism $p:E \to K$  is a mininal right \emph{(}$\ACM  \X$\emph{)}-approximation.
	 		\item[(b)] The morphism $i:K \to L \langle \x \rangle$   is a mininal left \emph{(}$\ACM \X$\emph{)}-approximation.	
	 	\end{itemize}
	 \end{proposition}
	 \begin{proof}
	 	We only prove (a), since (b) can be shown by a similar argument.  
	 	
	 	(a) First we show that the morphism $p:E\to K$ is right minimal.
	 	Let $f\in \End E$ such that the right square of the following diagram 
	 	\[
	 	\begin{tikzcd} 
	 		0\rar&	L(\w) \rar{\alpha}\dar{g}& E \dar{f} \rar{p}& K\dar[equals] \rar &0 \\
	 		0\rar&	L(\w)\rar{\alpha}&E \rar{p} & K \rar &0
	 	\end{tikzcd}
	 	\]
	 	commutes. Thus there exists a morphism $g:L(\w) \to L(\w)$ such that the left square commutes. Since the two exact rows of the diagram are nonsplit and $\End L(\w)=k$, $g$ is an isomorphism. Hence  $f$ is also an isomorphism. 
	 	
	 	Next we show that $p$ is a right $\ACM \X$-approximation of $K$. For any $M\in \ACM \X$, we apply $\Hom(M,-)$ to $\eta_1$ and obtain an exact sequence
	 	$$(M,E) \xrightarrow{p^{\ast}} (M,K)\to {}^{1}(M,L(\w))=0.$$
	 	Thus  $p^{\ast}:\Hom(M,E) \to \Hom(M,K)$ is surjective. This finishes the proof.
	 \end{proof}

	 	\begin{remark} Let $0\to K \to L \langle \x \rangle \xrightarrow{\gamma} L(\x) \to 0$ be an exact sequence. 	 By Proposition \ref{homological properties of $E$ and $K$ 1}, 
	 		 $\Ext^{1}(K,L(\y+\w))\simeq D\Ext^{1}(L(\y),K))=k$ holds for any $0\le \y  \le \vdelta$. Let $0\to L(\y+\w) \to M \to K \to 0$ be a nonsplit exact sequence in $\coh \X$. Then 
	 		 $$M\in \ACM \X \ \text{ if and only if } \ \y=0.$$ In fact, for $\y > 0$, one can  check that $\Ext^{1}(L,M)\simeq\Ext^{1}(L,K)=k$, a contradiction to the statement (\ref{ACM bundle eq.}) for $M$ being an ACM bundle.
	 		
	 	\end{remark}

	 	\begin{proposition} \label{uniqueness of org} Assume that there exists an exact sequence 
	\begin{equation}\label{2-uuexact}
		0 \to L(\w) \xrightarrow{}  F  
		\xrightarrow{} 
		L \langle \x \rangle \xrightarrow{\gamma} L(\x) \to 0
	\end{equation}
	 in $\ACM \X$. Then there is an isomorphism  $ F\simeq E_L\langle \x \rangle.$  
	\end{proposition}
	\begin{proof} Let $K:=\ker \gamma$. 
		First we show that the exact sequence (\ref{2-uuexact}) is nonsplit.  Assume for contradiction that it is split. By Lemma \ref{2-Yoneda}, there exists an object $X$ in $\coh \X$ fitting into the following commutative diagram 
		\[
		\begin{tikzcd} 
			& 0 \dar& 0\dar & 0\dar &\\
			0\rar &	L(\w)\rar\ar[equals]{d}&F\rar\dar&K \dar\rar & 0\\
			0\rar &	L(\w)\rar&X\rar\dar&L \langle \x \rangle \dar\rar & 0 \\
			& & L(\x)\ar[equals]{r} \dar & L(\x)\dar& \\
			&  & 0 & 0 & 
		\end{tikzcd}
		\]	
		with exact rows and columns in $\coh \X$.  By (\ref{ACM bundle eq.}), we have that both the middle row and column are split, a contradiction. Thus the exact sequence (\ref{2-uuexact}) is nonsplit.
		
		By  Proposition \ref{homological properties of $E$ and $K$ 1},  we obtain
		$\Ext^{1}(K,L(\w))\simeq D\Ext^{1}(L,K)=k.$ 
		Moreover, both of the exact sequences 
		$$0\to L(\w)\to E_L\langle \x \rangle \to K \to 0 \ \ \text{and} \ \ 0\to L(\w)\to F \to K \to 0$$
		are also nonsplit since the exact rows of the following diagram are nonsplit.
		This implies the commutativity of the following diagram. 
		\[
		\begin{tikzcd}[column sep=1.3em, row sep=0.5em]
			0\ar{rr}&&L(\w)\ar[rr,""]\ar[equals]{dd}&&F\ar[rr,""]\ar[dr]\ar[dd,""]&&L\langle \x \rangle\ar[rr, "\gamma"]\ar[equals]{dd}&&L(\x)\ar{rr}\ar[equals]{dd}&&0\\
			&&&&&K\ar{ur}\\
			0\ar{rr}&&L(\w)\ar[rr,"\alpha"]&&E_L\langle \x \rangle\ar[rr, "\ \ \ \ \ \ \beta"]\ar{dr}&&L\langle \x \rangle\ar[rr, "\gamma"]&&L(\x)\ar[rr]&&0\\
			&&&&&K\ar{ur}\ar[crossing over, equals]{uu}
		\end{tikzcd}
		\]	
		Hence we obtain the desired isomorphism.   	
	\end{proof}
	 As an immediate consequence of Propositions \ref{homological properties of $E$ and $K$ 1}, \ref{homological properties of $E$ and $K$ 3} and \ref{uniqueness of org}, we have the following result, giving the equivalent descriptions of $2$-extension bundles.
	 
	 \begin{theorem} \label{eq. des. 2-ext} Let $\xi: 0 \to L(\w) \xrightarrow{}  F  
	 	\xrightarrow{} L \langle \x \rangle \xrightarrow{\gamma} L(\x) \to 0$ be an exact sequence in $\coh \X$. Then the following conditions are equivalent. 
	 	\begin{itemize}
	 		\item[(a)]  $F \simeq E_L\langle \x \rangle $.
	 		\item[(b)]  $F \in \ACM \X$.
	 		\item[(c)]  $\xi$ is nonsplit.
	 	\end{itemize}
	 \end{theorem}

	 \begin{proposition}\label{shift preserve extension bundle}
	 	\begin{itemize}
	 		\item[(a)] Let $E$ be a $2$-extension bundle. Then  $\rk E=4$. 
	 		\item[(b)] There exists an isomorphism $$(E_L\langle \x \rangle)(\z)\simeq E_{L(\z)}\langle \x \rangle$$
	 		for any $0 \leq \x \leq \vdelta $ and $\z \in \L$.
	 	\end{itemize}	
	 \end{proposition}
	 \begin{proof}(a) This follows from the additivity of the rank function.
	 	
	 	(b)	Applying degree shift $(-\z)$ to the defining sequence  for $E_{L(\z)}\langle \x \rangle$, we obtain an exact sequence
	 	$0 \to L(\w) \to E_{L(\z)}\langle \x \rangle(-\z) \to L\langle \x \rangle \xrightarrow{\gamma} L(\x) \to 0$. Then we have $E_{L(\z)}\langle \x \rangle(-\z) \simeq E_L \langle \x \rangle$ by Proposition \ref{uniqueness of org}.	Thus the assertion follows.
	 \end{proof}
	 
	 	\subsection{$2$-coextension bundles}
	 In this section, we introduce the dual concept of $2$-extension bundles, called $2$-coextension bundles. 	   
	  For any $\x=\sum_{i=1}^4\lambda_i\x_i$ with $0\le \x \le \vdelta$, we consider the monomorphism 	
	 \begin{equation}
	 	L(\w) \xrightarrow{\gamma'}   \bigoplus_{1\le i \le 4}L(\w+(1+\lambda_i)\x_i),
	 \end{equation}
	 in $\coh \X$, where $\gamma':=(X_i^{\lambda_i+1})_{1\le i\le 4}$. For simplicity, we write 
	 \begin{align}
	 		 L_i[\x]:= L(\w+(1+\lambda_i)\x_i) \ \ \text{and}\ \ L[\x]:=\bigoplus_{1\le i\le4}L_i[\x].
	 \end{align}	  
	 Let $\mu_{2}: 0 \to L(\w) \xrightarrow{\gamma'} L[\x] \to C_{\x,L} \to 0$ be an exact sequence in $\coh \X$.  We usually omit the subscript of $C_{\x,L}$ below. Note that $\Ext^{1}(L(\x),C)\simeq \Ext^{2}(L(\x),L(\w))=k$  by (\ref{Auslander-Reiten-Serre duality}). Let $\mu_1: 0\to C \to F \to L(\x)  \to 0 $ be a nonsplit exact sequence. By connecting the two sequences $\mu_1$ and $\mu_2$, we finally obtain an exact sequence
	 \begin{equation} \label{2-coextension bundles}
	 \mu: \quad	0 \to L(\w)\xrightarrow{\gamma'} L[\x] \xrightarrow{\beta'} F \xrightarrow{\alpha'} L(\x) \to 0
	 \end{equation}
	 in $\coh \X$. In this case, the term $F_L \langle \x \rangle:=F$ of the sequence, which is uniquely defined up to isomorphism, is called the \emph{$2$-coextension bundle} gived by the data $(L,\x)$. If $L=\OO$, then we just write $F\langle \x \rangle$. For $\x=0$, the term $F_L:=F_L\langle 0 \rangle$ is called the \emph{$2$-coAuslander bundle} associated with $L$. 
	 
	 \begin{proposition}
	 	\label{uniqueness of org co} 
	 	Assume that there exists an exact sequence 
	 	\begin{equation*}
	 		0 \to L(\w) \xrightarrow{\gamma'}  L[\x]  
	 		\xrightarrow{} 
	 		E \xrightarrow{} L(\x) \to 0
	 	\end{equation*}
	 	 in $\ACM  \X$. Then  there is an isomorphism $ E\simeq F_L\langle \x \rangle$. 
	 \end{proposition}
	 \begin{proof}
	 	The proof is parallel to that of Proposition \ref{uniqueness of org}.
	 \end{proof}
	 
	 Recall that \emph{vector bundle duality} $(-)^{\vee}: \vect\X \to \vect\X$, $X\mapsto \Hhom{}{X}{\OO}$, defined by the sheaf Hom functor, sends line bundle $\OO(\x)$ to $\OO(-\x)$ for any $\x \in \L$.
	 
	 \begin{proposition}\label{vector bundle duality}
	 	Let $L$ be a line bundle and  $0\le \x \le \vdelta$. Then
	 	$$E_L \langle \x \rangle^{\vee}=F_{L^{\vee}}\langle \x \rangle(-\x-\w).$$
	 	In particular, for $L=\OO$, we obtain $E \langle \x \rangle^{\vee}=F\langle \x \rangle(-\x-\w).$
	 \end{proposition}
	 \begin{proof} Let $\eta: 0\to L(\w) \to E_L \langle \x \rangle \to L\langle \x \rangle \xrightarrow{\gamma} L(\x) \to 0$ be the defining sequence for the $2$-extension bundle $E_L \langle \x \rangle$. Applying vector bundle duality and the degree shift by $\x+\w$ to the sequence $\eta$, we obtain an exact sequence
	 	$$0\to L^{\vee}(\w)  \xrightarrow{\gamma'} L^{\vee}[\x]  \to E_{L^{\vee}} \langle \x \rangle(\x+\w) \to L^{\vee}(\x) \to 0.$$	
	 	By  Proposition \ref{uniqueness of org co}, we have 	$E_L \langle \x \rangle^{\vee}=F_{L^{\vee}}\langle \x \rangle(-\x-\w)$.
	 \end{proof}
	 
	\begin{remark}By the preceding lemma, like to the $2$-extension bundle, we note that each $2$-coextension bundle 
		is also an exceptional ACM bundle of rank four. Moreover, we can obtain dual version of Theorem \ref{eq. des. 2-ext}. 
	\end{remark}

	 \subsection{$2$-(co)Auslander bundles}
	 In this section,  we will give a description on how the $2$-Auslander bundles (resp. $2$-coAuslander bundles) are connected to the line bundles by the left (resp. right) minimal almost split morphisms in $\ACM \X$.

	 As before, let $ \eta: 0 \to L(\w) \xrightarrow{\alpha}  E_L 
	 \xrightarrow{\beta} 
	 L\langle 0 \rangle \xrightarrow{\gamma} L \to 0$ and 	$\mu: 0\to L(\w) \xrightarrow{\gamma'} L[0] \xrightarrow{\beta'} F_L \xrightarrow{\alpha'} L \to 0 $ 	 
	  be the defining sequences for the $2$-Auslander bundle $E_L$ and $2$-coAuslander bundle $F_L$, respectively. 
	  	  	  
	  Recall that   $\rho$ denotes the composition $\DDD^{\bo}(\mod^{\L}R)\to\DDD_{\rm sg}^{\L}(R)\xrightarrow{\sim} \underline{\CM}^{\L}R$ and $\pi:{\CM}^{\L}R\xrightarrow{\sim} \ACM\X$ denotes the sheafification. We identify objects in $\underline{\CM}^{\L}R$ with objects in ${\CM}^{\L}R$ without nonzero projective direct summands. Denote by $\pi(\rho(k))$ the image of $\rho(k)$ in $\coh \X$ under $\pi$, and by $[1]$  the suspension functor in $\underline{\ACM} \, \X$.		 	
	 
	 \begin{proposition}\label{sq.cor} For $L=\OO$, we have $E[2]=\pi(\rho(k))=F[1]$.
	 \end{proposition}
	 
	 \begin{proof} 
	 	 We only show the second equality, the first one can be shown similarly.
	 	 
	 	 Applying the graded global section functor 
	 	$$\Gamma:\quad\coh \X \to \Mod^{\L} R, \ \ X \mapsto \bigoplus_{\x \in \L} \Hom_{\X}(\OO(-\x),X)$$ to the sequence $\mu$,
	 	we obtain an exact sequence in $\mod^{\L}R$
	 	\begin{align*}  	\Gamma(\mu):\quad 0\to R(\w) \xrightarrow{\gamma'} \bigoplus_{1\le  i\le4}R(\w+\x_i) \to   M \to R \xrightarrow{q} k\to 0,
	 	\end{align*}
	 	where $M \in \CM^{\L} R$. Let $C:=\coker \gamma'$ and $\mathfrak{m}:=\ker q$. These yield triangles
	 	\begin{align*} 
	 		R(\w) &\to \bigoplus_{1\le  i\le4}R(\w+\x_i) \to C \to R(\w)[1],\\
	 		C &\to M \to \mathfrak{m} \to C[1],\\
	 		\mathfrak{m} &\to R \to k \to \mathfrak{m}[1]
	 	\end{align*} in $\DDD^{\bo}(\mod^{\L}R)$,  where we view every object in the triangles as stalk-complex. By standard properties of the Verdier quotient, we have 
	 	$C=0$, $M\simeq \mathfrak{m}$ and $k \simeq \mathfrak{m}[1]$ in $\DDD_{\rm sg}^{\L}(R)$, which implies $M[1] \simeq \rho(k)$  in $\underline{\CM}^{\L} R$. Note that the sequence $\mu$ is obtained back from $\Gamma(\mu)$ by applying the sheafification $\pi$. By the triangle equivalent $ \underline{\CM}^{\L} R\simeq \underline{\ACM} \, \X $,  we obtain $\pi(\rho(k))=F[1]$.	
	 \end{proof}
	 
	 As an immediate consequence,  the morphism $\alpha$ (resp. $\alpha'$) is a left (resp. right) minimal almost split morphism in $\ACM \X$ by a general result in \cite[Theorem 4.13]{HIMO} under the equivalence $\pi: \CM^{\L} R \xrightarrow{\sim} \ACM \X$.
	 
%

	\section{Basic properties of $2$-(co)extension bundles}
	\label{sec:Basic properties of 2-extension bundles}
	In this section, we establish the key features of $2$-(co)extension bundles. Most of major results of this paper are based on these properties.	
	The treatment of this section is parallel to Kussin-Lenzing-Meltzer's exposition of the properties of extension bundles  given in \cite[Section $4$]{KLM}.

	\subsection{Injective hulls and projective covers} 
	In this section, we will give the injective hulls and projective covers of $2$-extension bundles. This investigates the close relationship between $2$-extension  and $2$-coextension bundles. 	 		

	By using a similar argument as in the proof of \cite[Lemma 3.5]{I3}, we have the following observation. 
	Here, 	$\add X$ denotes the full subcategory of $\coh \X$ consisting of direct summands of finite direct sums of $X$.
 	\begin{lemma}\label{applying hom} Let $X\in \coh \X$ and $$0\to X_0 \to X_1\to X_2\to X_{3} \to 0$$
	be an exact sequence in $\coh \X$ with $X_i\in \add X$.
	\begin{itemize}			
		\item[(a)]   If $W\in \coh\X$ satisfies $\Ext^{1}(W,X)=0$, then we have an exact sequence
		$$0\to (W,X_{0})\to\cdots\to(W,X_{3})\to {}^2(W,X_{0})\to\cdots\to{}^2(W,X_{3})\to 0.$$
		\item[(b)] If $Y\in \coh\X$ satisfies $\Ext^{1}(X,Y)=0$, then we have an exact sequence
		$$0\to(X_3,Y)\to\cdots\to(X_0,Y) \to{}^2(X_3,Y)\to\cdots\to{}^2(X_0,Y)\to0.$$
	\end{itemize}
\end{lemma}

\begin{lemma}\label{ext1(E,E)=0} 
	Let  $L$ be a line bundle  and $0\le \y\le \x \le \vdelta$. Then  $$\Ext^{1}(E_L \langle\y \rangle,E_L \langle \x \rangle)=0.$$
 
\end{lemma}
\begin{proof}We factor the defining sequence (\ref{2-extension bundle exact}) for the $2$-extension bundle $E_L\langle \y \rangle$ into the two short exact sequences	
	$$\eta_1: 0\to L(\w)\to E_L\langle \y \rangle \to K_{\y} \to 0 \ \ \text{and} \ \ \eta_2: 0 \to K_{\y} \to L\langle \y \rangle \to L(\y) \to 0.$$
	By applying $\Hom(-,E_L\langle \x \rangle)$ to  $\eta_2$, we obtain an exact sequence 
	$$ {}^{1}(L\langle \y \rangle,E_L \langle \x \rangle) \to {}^{1}(K_{\y},E_L \langle \x \rangle) \to {}^{2}(L(\y),E_L \langle \x \rangle).$$ By Proposition \ref{homological properties of $E$ and $K$ 1}, the end terms of this sequence are zero, which implies that $\Ext^{1}(K_{\y},E_L \langle \x \rangle)=0$.
	Applying $\Hom(-,E_L \langle \x \rangle)$ to  $\eta_1$, we obtain an exact sequence
	$${}^{1}(K_{\y},E_L \langle \x \rangle) \to
	{}^{1}(E_L \langle \y \rangle,E_L \langle \x \rangle) \to {}^{1}(L(\w),E_L \langle \x \rangle)=0. $$			
	Therefore we have $\Ext^{1}(E_L \langle \y \rangle,E_L \langle \x \rangle)=0$. 		 	
\end{proof}

		\begin{proposition} \label{ext-pullback} 
	Assume $0\le \x \le \x+\y \le \vdelta$ and $\y=\sum_{i=1}^{4}e_i\x_i$ with $e_i \ge 0$.
	Consider $y:=\prod_{i=1}^4X_i^{e_i}$ as a morphism $L(\x)\to
	L(\x+\y)$. Then there exists a commutative diagram with exact rows	
	\begin{equation}\label{pullback and pushout diagram}
			\begin{tikzcd}[column sep=1.8em]
			\eta_{\x}:\quad\,\,	0\rar &	L(\w)\rar{\alpha}\arrow[d,equal]&E_L\langle \x \rangle \rar{\beta}\dar{y^{\ast}}&L\langle \x \rangle  \rar{\gamma}\dar{y'}&L(\x)\arrow[d,"y"] \rar& 0\,\\
			\eta_{\x+\y}:\,\,	0\rar&	L(\w)\rar{\alpha}&E_{L}\langle \x+\y \rangle \rar{\beta}&L\langle \x+\y \rangle\rar{\gamma}&L(\x+\y) \rar& 0,
		\end{tikzcd}
	\end{equation}	
	where $y':={\rm diag}(y/X^{e_1}_1,\dots, y/X^{e_4}_4)$, and $\Hom(E_L \langle \x \rangle, E_L \langle \x+\y \rangle)=k{y^{\ast}}$. Moreover, for any $i\neq j$ and $0 \le \x \le \x+\x_i+\x_j\le \vdelta$, the diagram below is commutative.
	\[
	\begin{tikzcd} 
		E_L\langle \x \rangle \rar{x_i^\ast}\dar{x_j^\ast}&E_L\langle \x+\x_i \rangle \dar{x_j^\ast} \\
		E_L\langle \x+\x_j \rangle\rar{x_i^\ast}&E_L\langle \x+\x_i+\x_j \rangle
	\end{tikzcd}
	\]	
\end{proposition}
\begin{proof}
	Let $K_{\x}$ be a kernel of the morphism $L\langle \x \rangle \to L(\x)$. 
	By commutativity of the right square of the diagram below,  there exists a morphism $y'':K_{\x} \to K_{\x+\y}$ making the  following diagram  commutative.
	\begin{equation}\label{right}
		\begin{tikzcd} 
		\eta_{\x,2}: \quad\,\,	0\rar&	K_{\x} \rar\dar{y''}& L\langle \x \rangle \dar{y'} \rar{\gamma}& 
			L(\x)\dar{y} \rar &0 \\
		\eta_{\x+\y,2}: \,\,	0\rar&	K_{\x+\y} \rar&L\langle \x+\y \rangle \rar{\gamma} & L(\x+\y) \rar &0
		\end{tikzcd}
	\end{equation}	
	 Construct a pullback of $\eta_{\x+\y,1}$ along the morphism $y''$
	\begin{equation}\label{left}
		\begin{tikzcd} 
		\eta_{\x,1}:\quad\,\,	0\rar&	L(\w) \rar\dar[equal]& M \dar \rar& 
			K_{\x}\dar{y''} \rar &0\, \\
		\eta_{\x+\y,1}:\,\,	0\rar&	L(\w) \rar&E_L\langle \x+\y \rangle \rar & K_{\x+\y} \rar &0.
		\end{tikzcd}
	\end{equation}	
			
	We claim that the exact sequence $\eta_{\x,1}$ does not split and thus  $M\simeq E_L \langle \x \rangle$. 
	Note that both $y$ and $y'$ are monomorphisms in $\coh \X$. Invoking the Snake Lemma to the diagram (\ref{right}), $y''$ is also monomorphism and we obtain an exact sequence
	$$
		0 \to \coker y'' \to \coker y'  \to \coker y \to 0.
	$$
	It is easy to check that the spaces $\Ext^{i}(L,\coker y)$ and $\Ext^{i}(L,\coker y')$ vanish for any $i\ge 0$, which implies $\Ext^{i}(L,\coker y'')=0$ for any $i\ge 0$. Applying $\Hom(-,L(\w))$ to the exact sequence $0\to K_{\x} \to K_{\x+\y} \to \coker y'' \to 0$, we obtain an exact sequence
		\begin{align*}
		  {}^{1}(\coker y'',L(\w)) \to {}^{1}(K_{\x+\y},L(\w)) \xrightarrow{y''_{\ast}} {}^{1}(K_{\x},L(\w)) \to {}^{2}(\coker y'',L(\w)).
	\end{align*}
	By Proposition \ref{homological properties of $E$ and $K$ 1} and (\ref{Auslander-Reiten-Serre duality}), both the spaces ${\Ext}^{1}(K_{\x+\y},L(\w))$ and ${\Ext}^{1}(K_{\x},L(\w))$ are 
	one dimensional.
	Using Auslander-Reiten-Serre duality, we have $${\Ext}^{2}(\coker y'',L(\w))\simeq D{\Hom}(L,\coker y'')=0.$$ 
	Thus the morphism $y''_{\ast}:\Ext^{1}(K_{\x+\y},L(\w))\to \Ext^{1}(K_{\x},L(\w))$ is an isomorphism, implying that the sequence $\eta_{\x,1}$ does not split. Connecting the two diagrams (\ref{right}) and (\ref{left}), we obtain the diagram (\ref{pullback and pushout diagram})
	 and thus  the first assertion follows.
	
	Concerning the second assertion, by Lemma \ref{applying hom},
	we apply $\Hom(-,L(\x+\y))$ to the sequence $\eta_{\x}$ and  obtain an exact sequence
	$$0\to (L(\x),L(\x+\y))\to (L\langle \x \rangle ,L(\x+\y))\to (E_L \langle \x \rangle,L(\x+\y)) \to (L(\w),L(\x+\y))=0.$$
	It is easy to verify that $\Hom(L(\x),L(\x+\y))=k$ and $\Hom(L\langle \x \rangle ,L(\x+\y))=k^4$ hold as $k$-vector spaces, which yields $\Hom(E_L \langle \x \rangle,L(\x+\y))=k^3$. By applying $\Hom(-,L\langle \x+\y \rangle)$ to  $\eta_{\x}$, it is  straightforward to verify that 
$\Hom(E_L \langle \x \rangle,L\langle \x+\y \rangle)\simeq \Hom(L\langle \x \rangle ,L\langle \x+\y \rangle)=k^4.$ 
	Invoking Lemma \ref{ext1(E,E)=0}, we can apply $\Hom(E_L \langle \x \rangle,-)$ to the sequence  $\eta_{\x+\y}$ and obtain an exact sequence 
	\begin{align*}
		0 =(E_L \langle \x \rangle,L(\w))&\to (E_L \langle \x \rangle ,E_L \langle \x+\y \rangle)\to (E_L \langle \x \rangle,L\langle \x+\y \rangle) \\ &\to (E_L \langle \x \rangle,L(\x+\y)) \to {}^{2}(E_L \langle \x \rangle,L(\w))=D(L,E_L \langle \x \rangle)=0.
	\end{align*}
	By the additivity of dimension, we have $\Hom(E_L \langle \x \rangle,E_L \langle \x+\y \rangle)=ky^{\ast}$ since $y^{\ast}$ is a nonzero morphism. By a similar argument, one can show that $\Hom(E_L \langle \x \rangle,E_L \langle \x+\x_i+\x_j \rangle)=k$ implying commutativity $x_i^{\ast}x_j^{\ast}=x_j^{\ast}x_i^{\ast}$.	
\end{proof}

	 \begin{theorem}[Injective Hulls] Let $E_L\langle \x \rangle$ be a $2$-extension bundle for $\x=\sum_{i=1}^4\lambda_i\x_i$ with $0\le \x \le \vdelta$. Then its injective hull
		 	$\mathfrak{I}(E_L\langle \x \rangle)$ is given by
		 	\begin{equation}\label{inj.hull}
			 		\mathfrak{I}(E_L\langle \x \rangle)=   \bigoplus_{1\le i \le 4} L(\w+(1+\lambda_i)\x_i) \oplus \bigoplus_{1\le i \le 4} L(\x-(1+\lambda_i)\x_i).
			 	\end{equation}
		 	Furthermore, the eight line bundle summands $\{L_i\}_{i=1}^8$ of  $\mathfrak{I}(E_L\langle \x \rangle)$ are mutually
		 	\emph{Hom-orthogonal}, that is, they satisfy
		 	\begin{equation}\label{eq:hom-orth}
			 		\Hom(L_i,L_j)=\begin{cases}k &  \ \textrm{if} \  i=j \\  
				 			0 &   \	\textrm{if}\  i\neq j.   \end{cases}
			 	\end{equation}
		 \end{theorem}
	 \begin{proof} Note that the existence of injective hulls is preserved under the equivalence $\pi:\CM^{\L} R\to\ACM \X$ as Frobenius categories. It thus remains to deal with the injective hulls. We proceed by induction on $n:=\sum_{i=1}^{4} \lambda_i$. 
	 	
	 	For $n=0$, the assertion reduces to the claim that for the $2$-Auslander bundle $E=E_L$ given by the sequence $0 \to L(\w)\xrightarrow{\alpha} E_L \xrightarrow{\beta} L\langle 0 \rangle \xrightarrow{\gamma} L \to 0$, we have $\mathfrak{I}(E)=L\langle 0 \rangle \oplus \bigoplus_{i=1}^{4} L(\w+\x_i)$. 		 	
		 We show that the morphism $$\iota_E:E\xrightarrow{(\beta,(x_i'))} L\langle 0 \rangle \oplus \bigoplus_{1\le i\le4} L(\w+\x_i)$$ is the injective hull of $E$ in $\ACM \X$, where 
		 the component of 
		 the morphism $x_i':E \to L(\w+\x_i)$ for $1\le i \le 4$, is induced by the morphism $x_i:L(\w) \xrightarrow{X_i} L(\w+\x_i)$ since $\alpha$ is a left almost split morphism.
		  Let $h:E\to L'$ be a morphism to a line bundle. Then $h'=h\alpha$ is not an isomorphism since $\alpha$ is not a split monomorphism.
		 It follows that $h'$ has a presentation $h':=\sum_{i=1}^{4}h_ix_i$, where $x_i$ is the morphism $x_i:L(\w) \xrightarrow{X_i} L(\w+\x_i)$ and $h_i:L(\w+\x_i) \to L'$ for each $1\le i \le 4$. Thus we obtain the following commutative diagram.
		 \[
		 	\begin{tikzcd} 
		 		0\rar &	L(\w)\rar{\alpha}\ar[d,"(x_i)"'] &E \rar{\beta}\dar{h} \ar[ld, "(x_i')"']&L\langle 0 \rangle \\
		 		&	\bigoplus_{i=1}^{4} L(\w+\x_i) \ar[r,"(h_i)"]&L' & 
		 		\end{tikzcd}		 
		\]  
		
		By the canonical factorization of  $\beta$, we have $\beta=ip$, where 
		 $K:=\Im \beta$, $p: E \to K$ and $i:K \to L \langle 0 \rangle$. Then  $(h-\sum_{i=1}^{4}h_ix_i')\alpha=0$ implies that $h-\sum_{i=1}^{4}h_ix_i'$ factor through $p$, that is, there is a morphism $g:K\to L'$ such that  $h-\sum_{i=1}^{4}h_ix_i'=gp$.  
		 By Proposition \ref{homological properties of $E$ and $K$ 4}, the morphism $g$ factors through $i$, that is, there exists a morphism $q:L \langle 0 \rangle\to L'$ such that $g=qi$. 
		 Thus we have $h-\sum_{i=1}^{4}h_ix_i'=qip=q\beta$. Hence $h$ factor through $(\beta,(x_i'))$.
		 The Hom-orthogonal (\ref{eq:hom-orth}) is easy to show. Minimality of $\iota_E$ follows from the fact that the line bundle constituents of $\mathfrak{I}(E)$ are mutually Hom-orthogonal. This finishes the claim for $n=0$.
		 	
		Considering the induction step, we assume $0\le \x \le \x+\x_i\le \vdelta$ and the formula (\ref{inj.hull}) holds for $\x$. 
		 Then by Proposition \ref{ext-pullback}, for any line bunlde $L'$, we have the following diagram 
		 \[
		 \begin{tikzcd}[column sep=small, row sep=tiny]
		 	&E_L\langle \x \rangle\ar[rr]\ar[dd,"x_i^{\ast}"]\ar{dl}&&  \mathfrak{I}(E_L \langle \x \rangle) \ar{dd}\ar{dlll}\\
		 	L'\ar[equal]{dd}\\
		 	&E\langle \x+\x_i \rangle\ar[rr]\ar[dl]&&\mathfrak{I}(E_L \langle \x+\x_i \rangle) \ar[dlll]\\
		 	L'
		 \end{tikzcd}
		 \]		where the squares and the upper triangle are commutative. Since the cokernel of the morphism $x_i^{\ast}$ has rank zero, we have $\Hom(\coker (x_i^{\ast}),L')=0$. Thus
		  one can check that the lower triangle also commutes and thus we show the induction step.	
		\end{proof}
		
		 Let $E$ be an indecomposable ACM bundle and $\mathfrak{I}(E)$ be its injective hull in the Frobenius category $\ACM \X$.
		Then there is an exact sequence $0\to E \to \mathfrak{I}(E) \to E[1]\to 0$, where $[1]$ denotes the suspension functor in $\underline{\ACM} \, \X$. The following result explains that the $2$-coextension bundles are nothing but 
		 the corresponding $2$-extension bundles by the action of the suspension functor in $\underline{\ACM} \, \X$.
		
		\begin{proposition} \label{dual ext-def} We have $F_L\langle \x \rangle \simeq E_L\langle \x \rangle[1]$. Moreover,
			there is a commutative diagram with exact rows
			\begin{equation}{
					\begin{tikzcd} \label{pullback and pushout diagram 2}
						0\rar &	L(\w)\rar{\alpha}\arrow[d,equal]&E_L\langle \x \rangle \rar{\beta}\dar&L\langle \x \rangle  \rar{\gamma}\dar&L(\x)\arrow[d,equal] \rar& 0\,\\
						0\rar&	L(\w)\rar{\gamma'}&L[\x]\rar{\beta'}&F_L\langle \x \rangle\rar{\alpha'}&L(\x) \rar& 0.
				\end{tikzcd}}
			\end{equation}
		\end{proposition}
		\begin{proof} The exactness of the sequence
			$0 \to E_L\langle \x \rangle \to \mathfrak{I}(E) \to E_L\langle \x \rangle[1] \to 0$ implies that the following diagram
			\[
			\begin{tikzcd} 
				E_L\langle \x \rangle \rar\dar&L\langle \x \rangle \dar \\
				L[\x]\rar&E_L\langle \x \rangle[1]
			\end{tikzcd}
			\]
			is a both pullback and pushout diagram. Thus we have a commutative diagram 
			\begin{equation}{
					\begin{tikzcd} \label{pullback and pushout diagram 3}
						0\rar &	L(\w)\rar{\alpha}\arrow[d,equal]&E_L\langle \x \rangle \rar{\beta}\dar&L\langle \x \rangle  \rar{\gamma}\dar&L(\x)\arrow[d,equal] \rar& 0\,\\
						0\rar&	L(\w)\rar&L[\x]\rar&E_L\langle \x \rangle[1]\rar&L(\x) \rar& 0.
				\end{tikzcd}}
			\end{equation}		
			The commutativity of the leftmost square of (\ref{pullback and pushout diagram 3}) yields that the morphism $L(\w) \to L[\x]$ is given by $(X_i^{\lambda_i+1})_{1\le i\le 4}$.  By Proposition \ref{uniqueness of org co}, we have $F_L\langle \x \rangle \simeq E_L\langle \x \rangle[1]$ and thus we obtain the commutative diagram (\ref{pullback and pushout diagram 2}).
		\end{proof}

		\begin{proposition}\label{susp.ext.} 
			Let $L$ be a line bundle and $\x=\sum_{i=1}^4\lambda_i\x_i$ with $0\le \x \le \vdelta$. 
			Then for any $1\le i\le 4$, we have 
			\begin{align}\label{susp.}
				E_{L} \langle \x \rangle[1] = E_L \langle (p_i-2-\lambda_i)\x_i+\sum_{j\neq i}\lambda_j\x_j \rangle ((1+\lambda_i)\x_i).
			\end{align}
			In particular, we have $E_{L} \langle \x \rangle[2]=E_{L} \langle \x \rangle(\c)$ and $E_{L} \langle \x \rangle=E_{L} \langle \vdelta-\x \rangle(\x-\w-\c)$.  			
		\end{proposition}
		
		\begin{proof} Without loss of generality, we can assume $L:=\OO$.
			Fix $1\le i \le 4$, let $I=\{1,\dots,4\}\setminus\{i\}$ and $\y=(p_i-2-\lambda_i)\x_i+\sum_{j\neq i}\lambda_j\x_j$. Write $\y=\sum_{j=1}^{4}y_j\x_j$.
			By Proposition \ref{homological properties of $E$ and $K$ 2}, $G:=E \langle \x \rangle[1]$ and $H:=E\langle \y \rangle((1+\lambda_i)\x_i)$ are exceptional. By Proposition \ref{k0}, it suffices to show that the classes $[G]$ and $[H]$ in the Grothendieck group $K_0(\coh\X)$ are the same.
			Recall that $0\to \OO(\w) \to E \langle \x \rangle \to \OO\langle \x \rangle \to \OO(\x) \to 0$
			is the defining sequence for $E \langle \x \rangle$, putting $\ell_j=1+\lambda_j$ for $1\le j\le 4$, hence 
			\begin{align*}
				[G]&=[\mathfrak{I}(E\langle \x \rangle)]-[E\langle \x \rangle]=\sum_{j=1}^{4}[\OO(\w+\ell_j\x_j)]+[\OO(\x)]-[\OO(\w)].
			\end{align*}			
			On the other hand, we have
			\begin{align*}
				[H]&=[\OO(\w+\ell_i\x_i)]+\sum_{j=1}^{4}[\OO(\y-(1+y_j)\x_j+\ell_i\x_i)]-[\OO(\y+\ell_i\x_i)]\\
				&=[\OO(\w+\ell_i\x_i)]+[\OO(\x)]+\sum_{\substack{J \subset I \\ |J|=2}} [\OO(\w+\sum_{j\in J}\ell_j\x_j)]-[\OO(\w+\sum_{j\neq i}\ell_j\x_j)].
			\end{align*}
			Thus we have  the following equality 
			$$[H]-[G]=\sum_{J\subset I}(-1)^{|J|} [\OO(\w+\sum_{j\in J}\ell_j\x_j)].$$
			
			Notice that $(X_j)_{j\in I}$ is an $R$-regular sequence. This implies that $(X_j^{\ell_j})_{j\in I}$ is also an $R$-regular sequence. Hence the corresponding Koszul complex
			\begin{equation}\label{Koszul1}		
				0\to R(-\sum_{j\in I}\ell_j\x_j)\to \bigoplus_{\substack{J \subset I \\ |J|=2}}R(-\sum_{j\in J}\ell_j\x_j) \to \bigoplus_{j\in I} R(-\ell_j\x_j)\to R \to 0
			\end{equation}
			of $R$ is exact except in the rightmost position whose cohomology is $R/(X_j^{\ell_j})_{j\in I}$. Since this belongs to $\mod_0^{\L}R$, the image of (\ref{Koszul1}) in $\coh \X$ is exact. 
			Moreover, by applying vector bundle duality, we obtain an exact sequence 
			\begin{equation}\label{Koszul2}		
				0\to \OO\to \bigoplus_{j\in I} \OO(\ell_j\x_j) \to \bigoplus_{\substack{J \subset I \\ |J|=2}}\OO(\sum_{j\in J}\ell_j\x_j) \to \OO(\sum_{j\in I}\ell_j\x_j) \to 0
			\end{equation}
			in $\coh \X$. Applying degree shift by $\w$ to (\ref{Koszul2}), we obtain $[H]=[G]$ and thus the first assertion follows.			
			Applying (\ref{susp.}) repeatedly, we have the following equalities. 
			\begin{align*}
				E \langle \x \rangle[2]&=E \langle \y \rangle[1]((1+\lambda_i)\x_i)\\
				&=E \langle (p_i-2-y_i)\x_i+\sum_{j\neq i}y_j\x_j \rangle ((1+\lambda_i+1+y_i)\x_i) \\
				&=E \langle \x \rangle(\c), 	\\
				E \langle \x \rangle[4]&=E \langle (p_1-2-\lambda_1)\x_1+\sum_{j= 2}^{4}\lambda_j\x_j \rangle [3] ((1+\lambda_1)\x_1)\\
				&=E \big\langle \sum
				_{i=1}^{2} (p_i-2-\lambda_i)\x_i+\sum_{j=3}^{4}\lambda_j\x_j \big\rangle [2] (\sum_{i=1}^{2}(1+\lambda_i)\x_i)\\
				&=\cdots\\
				&=E \langle \vdelta-\x \rangle(\x-\w+\c).
			\end{align*}
			Hence we have $E \langle \x \rangle=E \langle \vdelta-\x \rangle(\x-\w-\c)$. This finishes the proof.	
		\end{proof}

		Two special cases are the following, which will be used frequently later.
		\begin{corollary} \label{susp. tri. cat.} 
			\begin{itemize}			
				\item[(a)] For any $1\le i \le 4$, $E_L\langle(p_i-2)\x_i\rangle[1]=E_L((p_i-1)\x_i).$
				\item[(b)] If one of weights $p_i=2$, then $E_L\langle \x \rangle[1]=E_L\langle \x \rangle(\x_i)$.  
			\end{itemize}
			
		\end{corollary}


	\begin{remark} 
		The preceding result states that the suspension functor of $\underline{\ACM} \, \X$ preserves $2$-extension bundles. Thus 
		we have that the	$2$-extension bundles coincide with $2$-coextension bundles. 
		But it does not always preserve $2$-Auslander bundles. 
		More precisely, it preserves $2$-Auslander bundles if and only if at least one weight is two, see Corollary \ref{ext.is Aus}.	 	
	\end{remark}

	 \begin{theorem}[Projective covers] \label{projective covers} Let $E_L\langle \x \rangle$ be a $2$-extension bundle for some line bundle $L$ and $0\le \x \le \vdelta$.  Then its projective cover
	 	$\mathfrak{P}(E_L\langle \x \rangle)$ are given by
	 	\begin{align} \label{pro.covers}
	 		\mathfrak{P}(E_L\langle \x \rangle)= L(\w) \oplus \Big(\bigoplus_{\substack{I \subset \{1,\dots,4\} \\ |I|=2}}L(\sum_{i\in I}(1+\lambda_i)\x_i+\w-\c)\Big)\oplus L(\x-\c),
	 	\end{align}
	 	where  $\x:=\sum_{i=1}^4\lambda_i\x_i$.  Furthermore, the line bundle summands $\{L_i\}_{i=1}^8$ of  $\mathfrak{P}(E_L\langle \x \rangle)$ are mutually
	 	Hom-orthogonal.	 	
	 \end{theorem}
	 \begin{proof} The Hom-orthogonal is easy to show.  It remains to deal with the projective covers. Fix $1\le i \le 4$, we let $\y=(p_i-2-\lambda_i)\x_i+\sum_{j\neq i}\lambda_j\x_j$.
	 	By Proposition \ref{susp.ext.}, we have $\mathfrak{P}(E_L\langle \x \rangle)=\mathfrak{I}(E_L\langle \x \rangle[-1])=\mathfrak{I}(E_L\langle \y \rangle)((1+\lambda_i)\x_i-\c)$.  Then we have
		 \begin{align*}
				 	L[\y]((1+\lambda_i)\x_i-\c)&=L(\w)\oplus \bigoplus_{j\neq i}L((1+\lambda_j)\x_j+(1+\lambda_i)\x_i+\w-\c),\\
				 	L\langle \y \rangle((1+\lambda_i)\x_i-\c) &= L(\sum_{j\neq i}\lambda_j\x_j+\lambda_i\x_i-\c) \oplus 
				 	\bigoplus_{j\neq i}L(-\x_i-\x_j+\sum_{l\notin\{i,j\}}\lambda_l\x_l)\\
				 	&=L(\x-\c) \oplus \bigoplus_{j\neq i}L(\sum_{l\notin\{i,j\}}(1+\lambda_l)\x_l+\w-\c).
			\end{align*}
		Hence $\mathfrak{P}(E_L\langle \x \rangle)$  has the form (\ref{pro.covers}).
	 \end{proof}

	 	\subsection{Comparison with $2$-(co)Auslander bundles}
	 	
	 	The aim of this section is to give explicit formulas of $\Hom$-spaces comparing with $2$-Auslander bundles 
	 	(resp. $2$-coAuslander bundles)  in $\underline{\ACM}\, \X$.
	 	 
	 	As before,  let $E$ be the $2$-Auslander bundle given by the defining sequence $0\to \OO(\w) \xrightarrow{\alpha} E \xrightarrow{\beta} \OO\langle 0 \rangle  \xrightarrow{\gamma} \OO\to 0$.
	 	
	 \begin{proposition}\label{Aus.to X}  Assume that $X$ is an indecomposable ACM bundle but not a line bundle. We write its projective cover in the form
	 	$$\mathfrak{P}(X)=\bigoplus_{j\in J}\OO(\y_j+\w)\xrightarrow{u=(u_j)}X.$$ Then $\underline{\Hom}(E(\x),X)$ is nonzero if and only if $\x=\y_i$ for some $i\in J$.
	 	If in addition $\{\OO(\y_j+\w)\}_{j\in J}$ are mutually Hom-orthogonal, then  $\underline{\Hom}(E(\x),X)=k$.		
	 \end{proposition}
	 		\begin{proof} Assume that there exists a nonzero morphism $h\in \underline{\Hom}(E(\x),X)$ for some $\x \in\L$. Let $\beta=ip$ be the canonical factorization of the morphism $\beta$, where
	 			$K:=\Im \beta$, $p: E(\x) \to K$ and $i:K \to \OO \langle 0 \rangle(\x)$. 
	 			Then $h':=h\alpha:\OO(\x+\w) \to X$ is also nonzero since otherwise $h$ would factor through $p$ and then   $h$ would factor through $\beta$ by Proposition \ref{homological properties of $E$ and $K$ 4}, a contradiction to our assumption on $h$. Since $u:\bigoplus_{j\in J} \OO(\y_j+\w)\to X$ is a projective cover of $X$, the morphism $h'$ can lift to $v$, that is, we have a factorization $h'=uv=\sum_{j\in J}u_jv_j$, where $v:=(v_j)$ and $u:=(u_j)$. It remains to show that at least one component $v_j:\OO(\x+\w)\to \OO(\y_j+\w)$ is an isomorphism. Assume for contradiction, we extends $v$ to $E(\x)$ since $\alpha$ is left almost split, yielding a morphism $\bar{v}:E(\x) \to \mathfrak{P}(X)$. Thus we have a  diagram as follows
	 			\[
	 			\begin{tikzcd} [column sep=2.8em, row sep=2em]
	 				\OO(\x+\w)\rar{\alpha}\ar[d,"v=(v_j)"']  &E(\x) \ar[dl,"\bar{v}"] \rar{\beta}\dar{h}&\OO\langle 0 \rangle(\x)  \\
	 				\bigoplus_{j\in J} \OO(\y_j+\w) \rar["u=(u_j)"']	&X &
	 			\end{tikzcd}		 
	 			\]
	 			where the square and the higher triangle are commutative. Since $(h-u\bar{v})\alpha=0$,  $h-u\bar{v}$ factors through $K$ and thus by Proposition \ref{homological properties of $E$ and $K$ 4}, it factors through $\OO \langle 0 \rangle(\x)$. Additionally, $u\bar{v}$ factors through $\mathfrak{P}(X)$ and then $h$ factors through $\OO \langle 0 \rangle(\x)\oplus \mathfrak{P}(X)$, a contradiction to our assumption.
	 				 				 			
	 			Conversely, fixing an index $i$ of $J$, we show that $\underline{\Hom}(E(\y_i),X)\neq 0$. Since $\alpha$ is left almost split in $\ACM \X$, we extends $u_i$ to the $2$-Auslander bundle $E(\y_i)$, implying a nonzero morphism $v:E(\y_i)\to X$. We claim that $v$ is nonzero in  $\underline{\Hom}(E(\y_i),X)$. Assume for contradiction that $v$ factors through $\mathfrak{P}(X)$. There exists a morphism $\bar{v}:E(\y_i)\to \mathfrak{P}(X)$ such that $v=u\bar{v}$. We put $\bar{\alpha}=\bar{v}\alpha$ and thus obtain the following commutative diagram
	 					\[
	 					\begin{tikzcd} [column sep=2.8em, row sep=2em]
	 								\OO(\y_i+\w)\rar{\alpha}\ar[d,"u_i"'] \ar[rd, "\bar{\alpha}=(\alpha_j)"]&E(\y_i) \rar{\beta}\dar{\bar{v}=(\bar{v}_j)}&\OO\langle 0 \rangle(\y_i)  \\
	 								X&\bigoplus_{j\in J} \OO(\y_j+\w) \ar[l," u=(u_j)"'].&
	 							\end{tikzcd}		 
	 					\]
	 					Consider the morphism $u_i=u\bar{\alpha}=\sum_{j\in J} u_j{\alpha_j}$, where $\alpha_j=\bar{v}_j\alpha$ for each $j \in J$. We distinguish two cases. If $\alpha_i=0$, then we have $u_i=\sum_{j\neq i} u_j{\alpha_j}$, a contradiction to the minimality of the projective cover $\mathfrak{P}(X)$. Otherwise $\alpha_i\neq0$, then $\alpha_i=\bar{v}_i\alpha$ is an automorphism of $\OO(y_i+\w)$, 	yielding that the morphism $\alpha$ is a split monomorphism, a contradiction. Thus we have shown that the nonzero morphism $v:E(\y_i)\to X$ does not factor through a direct sum of line bundles. This finishes the proof.	 						 					 	
	 				\end{proof}

	 As an immediate consequence, we have the following result, describing that there are enough $2$-Auslander bundles in the following sense.
	 
	 \begin{corollary}\label{Aus.coro} Assume that $X$ is an indecomposable ACM bundle but not a line bundle. Then there exists a $2$-Auslander bundle $E(\x)$ such that $\underline{\Hom}(E(\x),X)\neq 0$.		
	 \end{corollary}
	 
	 We put $\bar{x}_{ij}:=\c-\x_i-\x_j$  for any $1\le i< j \le 4$.
	 
	 \begin{corollary} \label{compare 2-Aus} 
	 	Let $E_L \langle \x \rangle$ be a $2$-extension bundle for some
	 	line bundle $L$ and $\x=\sum_{i=1}^{4}\lambda_i\x_i$ with $0\le \x \le \vdelta$. Put  $E_L:=E_L\langle 0 \rangle$. 
	 	Then  $\underline{\Hom}(E_L(\y),E_L \langle \x \rangle)\neq 0$ if and only if $\y$ is one of 	$0$, $\x-\c-\w$ or $(1+\lambda_i)\x_i+(1+\lambda_j)\x_j-\c$ with $1\le i< j \le4$. 
	 	Moreover in this case $\underline{\Hom}(E_L(\y),E_L \langle \x \rangle)=k$.
	 		 	
	 In particular, for a $2$-Auslander bundle $E$, we have $\underline{\Hom}(E,E(\x))\neq 0$ if and only if $\x$ is one of $0$, $\c+\w$ or $\bar{x}_{ij}$ for $1\le i< j \le 4$. In this case $\underline{\Hom}(E,E(\x))=k$.	
	 \end{corollary}
	 \begin{proof} This is immediate from Proposition \ref{Aus.to X} and Theorem \ref{projective covers}.	 
	 \end{proof}	
	
	Recall that for a subclass $\CC$ of a triangulated category $\mathcal{T}$, the \emph{right perpendicular category} $\CC^{\perp}$ of $\CC$ consists of all objects $X$ of $\mathcal{T}$ such that $\Hom(C,X[n])=0$ holds for any $C\in\CC$ and $n\in \Z$.

	\begin{corollary}\label{consists of all 2-Aus}  Let $\CC$ be the triangulated subcategory of $\underline{\ACM} \, \X$ containing all $2$-Auslander bundles. Then we have $\CC^{\perp}=0$.		
	\end{corollary}
	\begin{proof}Assume that $X$ is an indecomposable ACM bundle but not a line bundle. By Corollary  \ref{Aus.coro}, there is a $2$-Auslander bundle $E(\x)$ such that $\underline{\Hom}(E(\x),X)\neq 0$. Since $\CC$ contains all $2$-Auslander bundles, 	this shows $\CC^{\perp}=0$.		
	\end{proof}

	\begin{remark} \label{comp. 2-coAus}
		Let $0 \to \OO(\w) \to \OO[0] \to F \to \OO \to 0$ be  the defining sequence for the $2$-coAuslander bundle $F$. By vector bundle duality and Lemma \ref{vector bundle duality}, we can obtain dual versions of Proposition \ref{Aus.to X} and its corollaries. With the above notation we obtain, for instance, that $\underline{\Hom}(X,F(\x))\neq 0$ if and only if $\OO(\x)$ is a direct summand of the injective hull of $X$.		
	\end{remark}

	 \subsection{A distinguished triangle}

	 \begin{proposition}\label{triangle}  Assume that $0\le \x \le \x+\x_i \le \vdelta$ holds for some $1\le i\le 4$. Let $L$ be a line bundle.
	 	 Then there exists a triangle	in $\underline{\ACM} \, \X$	
	 	\begin{align}
	 		E_L \langle \x \rangle \xrightarrow{v_i} E_L \langle \x+\x_i \rangle \to E_L \langle \x-\lambda_i\x_i \rangle((1+\lambda_i)\x_i)\to E_L \langle \x \rangle[1],
	 	\end{align}
	 	 where $\x:=\sum_{i=1}^{4}\lambda_i\x_i$ and the morphism $v_i:=x_i^{\ast}$ is given in Proposition {\ref{ext-pullback}}. 
	 \end{proposition}	
	 \begin{proof} The cone of $v_i$ can be calculated from the exact sequence
	 	$$0 \to E_L\langle \x \rangle \to E_L\langle \x+\x_i \rangle \oplus \mathfrak{I}(E_L\langle \x \rangle) \to C \to 0. $$
	 	This is the sequence associated with the following pushout diagram 
	 	\[
	 	\begin{tikzcd} 
	 		&0 \dar  & 0 \dar  \\
	 	\xi:\quad	0\rar&	E_L \langle \x \rangle \rar{v_i}\dar{\iota_E}& E_L \langle \x+\x_i \rangle \dar \rar& U_i \dar[equals] \rar &0 \\
	 	\zeta:\quad	0\rar&	\mathfrak{I}(E_L\langle \x \rangle)\rar \dar &C \rar \dar & U_i \rar  &0 \\
	 		&E_L \langle \x \rangle[1] \dar \rar[equals] & E_L \langle \x \rangle[1] \dar  \\
	 		&0   & 0   
	 	\end{tikzcd}
	 	\]
	 	in $\coh \X$, where $U_i:=\coker v_i$ and $C$ belongs to $\ACM \X$. 
	 	Note that the injective hulls of $E_L\langle \x \rangle$ and $E_L \langle \x+\x_i \rangle$ have four common direct summands of line bundles $L_i \langle \x \rangle$ and $L_j[\x]$ for $j\neq i$. 
	 	This implies that $\Ext^{1}(U_i,L_i \langle \x \rangle)=0$ and $\Ext^{1}(U_i,L_j [\x])=0$ for $j\neq i$, yielding that $C$ also has these four line bundles as direct summands. 
	 	After canceling these common line bundle factors from $C$ and $\mathfrak{I}(E_L\langle \x \rangle)$, we obtain an exact sequence
	 	\begin{align}
	 		0 \to E_L \langle \x \rangle \to E_L \langle \x+\x_i \rangle\oplus L_i[\x] \oplus \bigoplus_{j\neq i} L_j\langle \x \rangle\to \underline{C} \to 0.
	 	\end{align}	
	 	
	 	We refer to Appendix \ref{excep C}  for the proof of exceptionality of $\underline{C}$  in $\coh \X$. Then we let $G:=E_L \langle \x-\lambda_i\x_i \rangle((1+\lambda_i)\x_i)$ and $H:=\underline{C}$.
	 	By Proposition \ref{k0}, it remains to show that the classes $[G]$ and $[H]$ in the Grothendieck group $K_0(\coh\X)$ are the same. By the defining sequence for $G$, putting $\ell_j=1+\lambda_j$ for $1\le j \le 4$, we have
	 	\begin{align*}
	 		[G]&= [L(\w+\ell_i\x_i)]+[L(\x)]+\sum_{j\neq i}[L(\x+\x_i-\ell_j\x_j)]-[L(\x+\x_i)], \\
	 		[H]&=[E_L \langle \x+\x_i \rangle] -[E_L \langle \x \rangle] + [L(\w+\ell_i\x_i)]+\sum_{j\neq i}[L(\x-\ell_j\x_j)]. 
	 	\end{align*}
	 	By substituting the following equalities
	 	\begin{align*}
	 		[E_L \langle \x+\x_i \rangle]&=[L(\x-\x_i)]+\sum_{j\neq i}[L(\x+\x_i-\ell_j\x_j)]+[L(\w)]-[L(\x+\x_i)], \\
	 		[E_L \langle \x \rangle]&=[L(\x-\x_i)]+\sum_{j\neq i}[L(\x-\ell_j\x_j)]+[L(\w)]-[L(\x)]
	 	\end{align*}
	 	into the right hand side of the equality of $[H]$, we obtain $[G]=[H]$. 
	 \end{proof}

	 An interesting special case is the following, which will be used frequently later. 
	 
	 \begin{corollary} \label{triangle2}
	 	For any  $1\le i \le 4$, let $\ell_i$ be an integer satisfying $1\le \ell_i\le p_i-2$. Then there exists a triangle
	 	\begin{align}
	 		E_L \langle (\ell_i-1)\x_i \rangle \to E_L \langle \ell_i\x_i \rangle \to E_L (\ell_i\x_i)\to E_L \langle (\ell_i-1)\x_i \rangle[1].
	 	\end{align}	
	 \end{corollary}

	As an application, we have the following result. 
	
		\begin{proposition} \label{triangulated subcategories} Let $L$ be a line bundle. Then the two triangulated subcategories of $\underline{\ACM} \, \X$ generated by the following objects, respectively, coincide:
		\begin{itemize}			
			\item[(a)]  the $2$-extension bundles $E_L\langle \x \rangle$, $0\le \x \le \vdelta$. 
			\item[(b)]  the $2$-Auslander bundles $E_L(\x)$, $0\le \x \le \vdelta$.   
		\end{itemize}
	\end{proposition}

		\begin{proof} Without loss of generality, we can assume $L:=\OO$. Let $\CC$ be the triangulated subcategory generated by $E\langle \x \rangle$, $0\le \x \le \vdelta$. 
		Clearly, we have $E(0)=E\langle 0 \rangle \in \CC$. 
		   The proof is divided into the following four steps. Note that  we can permute the indices of $\{\ell_i\}_{1\le i \le 4}$  of each step below by symmetry.

		\emph{Step $1$}: By Corollary \ref{triangle2},  we have the  triangle
		\begin{align}\label{tri,sub,con}
			E \langle (\ell_{1}-1)\x_{1} \rangle \to E \langle \ell_{1}\x_{1} \rangle \to E (\ell_{1}\x_{1})\to.
		\end{align}
		Thus we have $E(\ell_{1}\x_{1})\in \CC$ for  $0\le \ell_{1} \le p_1-2$.  
		
		\emph{Step $2$}: By Proposition \ref{triangle}, we have  the  triangle
		\begin{align}\label{tri,sub,con 2}
			E \langle \ell_1\x_1+(\ell_2-1)\x_2 \rangle \to E \langle \ell_1\x_1+\ell_2\x_2 \rangle \to E \langle \ell_1\x_1 \rangle(\ell_2\x_2)\to,
		\end{align}
		yielding  $E\langle \ell_1\x_1 \rangle(\ell_2\x_2)\in \CC$. Applying $\ell_2\x_2$ to (\ref{tri,sub,con}), we have the triangle
		\begin{align*}
			E \langle (\ell_1-1)\x_1 \rangle(\ell_2\x_2) \to E \langle \ell_1\x_1 \rangle(\ell_2\x_2) \to E (\ell_1\x_1+\ell_2\x_2) \to.
		\end{align*}
	Thus we have $E (\ell_1\x_1+\ell_2\x_2) \in \CC$ for $0\le \ell_i \le p_i-2$.
		
		\emph{Step $3$}: By Proposition \ref{triangle}, we get the  triangle 
		\begin{align}\label{tri,sub,con 3}
			E \langle \ell_1\x_1+\ell_2\x_2+(\ell_3-1)\x_3 \rangle \to E \langle \sum_{i=1}^{3}\ell_i\x_i \rangle \to E \langle \ell_1\x_1+\ell_2\x_2 \rangle(\ell_3\x_3)\to,\end{align}
		yielding   $E \langle \ell_1\x_1+\ell_2\x_2 \rangle(\ell_3\x_3)\in \CC$. 
		Applying $\ell_3\x_3$ to (\ref{tri,sub,con 2}), we have the triangle
		\begin{align*}
			E \langle \ell_1\x_1+(\ell_2-1)\x_2 \rangle(\ell_3\x_3) \to E \langle \ell_1\x_1+\ell_2\x_2 \rangle(\ell_3\x_3) \to E \langle \ell_1\x_1 \rangle(\ell_2\x_2+\ell_3\x_3)\to,
		\end{align*}
		which implies that $E \langle \ell_1\x_1 \rangle(\ell_2\x_2+\ell_3\x_3)\in \CC$. Then  applying to (\ref{tri,sub,con}) by degree shift $\ell_2\x_2+\ell_3\x_3$, we obtain the  triangle
		\begin{align*}
			E \langle (\ell_1-1)\x_1 \rangle(\ell_2\x_2+\ell_3\x_3) \to E \langle \ell_1\x_1 \rangle(\ell_2\x_2+\ell_3\x_3) \to E(\sum_{i=1}^{3}\ell_i\x_i)\to.
		\end{align*}	 
	Thus we have $E (\ell_1\x_1+\ell_2\x_2+\ell_3\x_3) \in \CC$ for $0\le \ell_i \le p_i-2$.	
		
		\emph{Step $4$}: By Proposition \ref{triangle}, we have the  triangle
		\begin{align}\label{tri,sub,con 4}
			E \langle \sum_{i=1}^{3}\ell_i\x_i+(\ell_4-1)\x_4 \rangle \to E \langle \sum_{i=1}^{4}\ell_i\x_i \rangle \to E \langle \sum_{i=1}^{3}\ell_i\x_i \rangle(\ell_4\x_4)\to,\end{align}
			which implies that $E \langle \ell_1\x_1+\ell_2\x_2+\ell_3\x_3 \rangle(\ell_4\x_4)\in \CC$ for  $0\le \ell_i \le p_i-2$. Considering the triangles obtained from (\ref{tri,sub,con 3}) by twisting with $\ell_4\x_4$,  (\ref{tri,sub,con 2}) by twisting with $\ell_3\x_3+\ell_4\x_4$, and  (\ref{tri,sub,con}) by twisting with $\ell_2\x_2+\ell_3\x_3+\ell_4\x_4$, we finally obtain $E (\x) \in \CC$ for $0\le \x \le \vdelta$ and which concludes the proof.

		The converse inclusion can be shown by  similar considerations, by going the steps of the preceding proof backwards.
	\end{proof}

			\section{The correspondence of $2$-Extension bundles}
			\label{sec:The correspondence of 2-Extension bundles}
			
		 \subsection{The correspondence theorem}	
		In this section, we will establish 
		correspondences between  $2$-extension bundles, $2$-coextension bundles and
		 an important class of  Cohen-Macaulay modules studied in \cite{HIMO}.

		Let $\s:=\sum_{i=1}^{4}\x_i$. Then for each element $\vell=\sum_{i=1}^{4}\ell_i\x_i\in [\s,\s+\vdelta]$, we put
		$$E^{\vell}:=R/(X_i^{\ell_i}\mid 1\le i\le 4)\in \mod^{\L}R \ \ \text{and} \ \ U^{\vell}:=\rho(E^{\vell}),$$
		where $\rho$ denotes the composition $\DDD^{\bo}(\mod^{\L}R)\to\DDD_{\rm sg}^{\L}(R)\xrightarrow{\sim} \underline{\CM}^{\L}R$.
		
		\begin{theorem} \label{correspondence of 2-Extension bundles}
			Let $R$ be an $\L$-graded quadrangle singularity of type $(p_1,\dots,p_4)$, and
			$\X$ be the corresponding GL projective space. 
			Then for each $\vell \in [\s,\s+\vdelta]$,    
			$$\pi(U^{\vell})= E\langle \s+\vdelta-\vell \rangle(-\w).$$
			Conversely, each $2$-extension bundle is of this form, up to degree shift.		
		\end{theorem}
		\begin{proof}		By an easy calculation, the defining sequence of $E\langle \s+\vdelta-\vell \rangle(-\w)$ is 
			\begin{align*} 
				\eta:\quad	0\to \OO \to E\langle \s+\vdelta-\vell \rangle(-\w) \to \bigoplus_{1\le  i\le4} \OO(2\c-\vell+\ell_i\x_i)\xrightarrow{\gamma} \OO(3\c-\vell)\to 0.
			\end{align*}
			Applying the graded global section functor 
			$$\Gamma:\quad\coh \X \to \Mod^{\L} R, \ \ X \mapsto \bigoplus_{\x \in \L} \Hom_{\X}(\OO(-\x),X)$$ to the sequence $\eta$,
			we obtain an exact sequence $\Gamma(\eta)$ in $\mod^{\L}R$
			\begin{align*}  	 0\to R \to M \to \bigoplus_{1\le  i\le4} R(2\c-\vell+\ell_i\x_i) \xrightarrow{\gamma} R(3\c-\vell) \xrightarrow{q} E^{4\c-\vell}(3\c-\vell)\to 0,
			\end{align*}
			 where $M\in \CM^\L R$. Let $K:=\ker \gamma$ and $\mathfrak{n}:=\ker q$.  These yield triangles 
			\begin{align*} 
				R &\to M \to K \to R[1],\\
				K &\to \bigoplus_{1\le  i\le4} R(2\c-\vell+\ell_i\x_i) \to \mathfrak{n} \to K[1],\\
				\mathfrak{n} &\to R(3\c-\vell) \to E^{4\c-\vell}(3\c-\vell)\to \mathfrak{n}[1]
			\end{align*} in $\DDD^{\bo}(\mod^{\L}R)$, where we view every object in the triangles as stalk-complex. By standard properties of the Verdier quotient, we have 
			$M\simeq K$, $  \mathfrak{n}\simeq K[1]$ and $ E^{4\c-\vell}(3\c-\vell)\simeq \mathfrak{n}[1]$ 
			in $\DDD_{\rm sg}^{\L}(R)$, implying $M[2] \simeq U^{4\c-\vell}(3\c-\vell)$  in $\underline{\CM}^{\L} R$. 
			Note that the sequence $\eta$ is obtained back from $\Gamma(\eta)$ by applying the sheafification $\pi$. 	By the triangle equivalent $ \underline{\CM}^{\L} R\simeq \underline{\ACM} \, \X $ and Proposition \ref{susp.ext.},  we have $$E\langle \s+\vdelta-\vell \rangle(-\w) = \pi(U^{4\c-\vell})(2\c-\vell).$$  
			Again by Proposition \ref{susp.ext.}, it follows that  $E\langle \s+\vdelta-\vell\rangle=E\langle \vell-\s\rangle(2\c-\vell)$. Therefore we have $\pi(U^{4\c-\vell})=E\langle \vell-\s\rangle(-\w)$. Replacing $\vell$ by $4\c-\vell$, we obtain $$\pi(U^{\vell})= E\langle \s+\vdelta-\vell \rangle(-\w).$$						
			Conversely, this is immediately  from Observation \ref{shift preserve extension bundle}(b).
		\end{proof}
		
	 As an immediate consequence of Proposition \ref{dual ext-def} and Theorem \ref{correspondence of 2-Extension bundles},  we obtain the following result.
		
		\begin{corollary} \label{bijections} There are bijections between:
			\begin{itemize}			
	\item[(a)]   The set  of $2$-extension bundles $E \langle \x \rangle$, $0 \leq \x \leq \vdelta;$
	\item[(b)]   The set of $2$-coextension bundles $F \langle \x \rangle$, $0 \leq \x \leq \vdelta;$
	\item[(c)]   The set of the graded Cohen-Macaulay modules $U^{\vell}$, $ \vell \in [\s, \s+\vdelta]$.
\end{itemize}	
		\end{corollary}

		 \subsection{Indecomposable rank four ACM bundles}
	In this section, we study the converse of Observation \ref{shift preserve extension bundle}(a). More precisely, we pose the following conjecture which is a higher dimensional version of \cite[Theorem 4.2]{KLM}.

	\begin{conjecture}\label{problem 1} 
		Let $\X$ be a GL projective space attached to an $\L$-graded quadrangle singularity of weight type $(p_1,\dots,p_4)$ with $p_i\ge 2$. 
		Then each indecomposable ACM bundle of rank four is a $2$-extension bundle. 
		Moreover, there is no rank two or three indecomposable ACM bundle. 	\end{conjecture}	
	
	The following result gives a position answer to Conjecture \ref{problem 1} when $\X$ is given by weight type $(2,a,b,c)$ with $a,b,c \ge 2$. We refer to \cite{KLM} for definitions of weighted projective lines and extension bundles.
	
	\begin{theorem} \label{rank 4 is 2-ext} Assume that $\X$ is given by weight type $(2,a,b,c)$ with $a,b,c \ge 2$. Then each  indecomposable ACM bundle $F$ of rank four is a $2$-extension bundle $F=E_L \langle \x \rangle$ for some line bundle $L$ and $0\le \x \le \vdelta$. Moreover, there is no rank two or three indecomposable ACM bundle.	
	\end{theorem}
	\begin{proof} Let $\Y$ be a weighted projective line with weight triple $(a,b,c)$.
		Using the correspondence in Theorem \ref{correspondence of 2-Extension bundles} and
		by \cite[Proposition 4.74]{HIMO}, there is a triangle equivalence $\varphi: \underline{\ACM}\, \Y \to \underline{\ACM}\, \X$, which sends extension bundles to $2$-extension bundles and the objects with rank $n$ to $2n$ in our setting, where we refer to \cite{KLM} for the distinguished exact structure of $\ACM \Y$ to be a Frobenius category. Following \cite[Theorem 4.2]{KLM}, each indecomposable  ACM bundle (vector bundle) of rank two in $\ACM \Y$ is extension bundle. Thus each indecomposable ACM bundle of rank four in $\ACM \X$ is $2$-extension bundle.	The latter assertion follows from that each indecomposable  object in $\underline{\ACM} \, \Y$ has at least rank two.
	\end{proof}
	
	\begin{corollary}Assume that $\X$ is given by weight type $(2,2,2,p)$ with $p \ge 2$. Then each indecomposable ACM bundle $F$ is either a line bundle or a $2$-extension bundle $F=E_L\langle \x \rangle$ for some line bundle $L$ and $0 \leq \x \leq (p-2)\x_4$.
		
	\end{corollary}

	\section{Tilting theory in the stable category of ACM bundles}
	\label{sec:Tilting theory in the stable category of ACM bundles}
	The aim of this section is to study tilting objects in the triangulated category $\underline{\ACM} \, \X$. We construct a tilting object in $\underline{\ACM} \, \X$ consisting only of $2$-(co)Auslander bundles. 
	Denote by $k\vec{\AA}_n$ the path algebra of the equioriented quiver of type $\AA_n$.

	\begin{theorem}[Tilting $4$-cuboid]\label{tilting $4$-cuboid}  	
		Let $\X$ be a GL projective space attached to an $\L$-graded quadrangle singularity of weight type $(p_1,\dots,p_4)$ with $p_i\ge 2$.  
		Then 
		$$T_{\rm cub}=\bigoplus_{0\leq\x\leq\vdelta}E_{L} \langle \x \rangle$$ is  a tilting object in $\underline{\ACM} \, \X$, called the \emph{tilting $4$-cuboid}, with endomorphism algebra  $$\underline{\End}(T_{\rm cub})^{\rm op}\simeq 
		\bigotimes_{1\le i \le 4} k\vec{\AA}_{p_i-1}.$$	
	\end{theorem}
	\begin{proof} Without lose of generality, we can assume $L=\OO$. Using the bijection given in Corollary \ref{bijections}, the  assertion
		 follows immediately from \cite[Theorem 4.76]{HIMO} under the triangle equivalence $\underline{\CM}^{\L} R \simeq \underline{\ACM} \, \X$. 		
	\end{proof}
	
		The following is the main result in this section. For this, we need the following piece of notation. For each $0\le \x\le \vdelta$ and write $\x=\sum_{i=1}^{4} \lambda_i\x_i$ in normal form, we let $ \sigma(\x):= \sum_{i=1}^{4} \lambda_i$. Recall that a finite dimensional algebra is called \emph{Nakayama algebra} if its indecomposable projective or injective modules are uniserial, that is, have a unique composition series. A natural class of such algebras are formed by the algebras $\vec{\AA}_{n}(m):=k\vec{\AA}_n/{\rm rad^m} \, k\vec{\AA}_n$, where $n,m \ge 1$.
	
		\begin{theorem}\label{main  theorem}
			Let $\X$ be a GL projective space attached to an $\L$-graded quadrangle singularity of weight type $(p_1,\dots,p_4)$ with $p_i\ge 2$. Let $E_L$ be a $2$-Auslander bundle for some line bundle $L$. 	 		
		Then 
		$$T=\bigoplus_{0\leq\x\leq\vdelta}E_L(\x)[-\sigma(\x)]$$
		is a  tilting object in $\underline{\ACM}\, \X$ with endomorphism algebra $$\underline{\End}(T)^{\rm op}\simeq \bigotimes_{1\le i \le 4} k\vec{\AA}_{p_i-1}(2). $$
	\end{theorem}
	
	We  need the following observations. Recall that $\s:=\sum_{i=1}^{4}\x_i$. 
	 \begin{proposition}\label{hom.rig} Let $L$ be a line bundle and $0\le \x,\y \le \vdelta$. Then we obtain
				\begin{eqnarray*}
					 	\underline{\Hom}(E_L(\x), E_L(\y)[n])=
						\begin{cases}
						 		k &  \text{if } n=\sigma(\x)-\sigma(\y) \text{ and } 0\le \x-\y \le \s,\\
						 		0& \text{otherwise.}
						 	\end{cases}
				\end{eqnarray*}
	 \end{proposition}
	\begin{proof}Without loss of generality, we let $L:=\OO$. Write $\x=\sum_{i=1}^{4}\lambda_i\x_i$. We divide the proof into two cases by considering $n$ is even or odd.
		
		\emph{Case $1$}: In this case $n$ is even and  write $n:=2m$.  By Proposition \ref{susp.ext.}, we have to deal with the expression $H:=\underline{\Hom}(E , E(\y-\x+m\c))$.	By Corollary \ref{compare 2-Aus}, $H$ is nonzero if and only if $\y-\x+m\c$ satisfies one of the following cases:
		$$\text{(a)} \ \y-\x+m\c=0, \ \ \text{(b)} \ \y-\x+m\c=\c+\w, \ \ \text{or} \ \ \text{(c)} \ \y-\x+m\c=\bar{x}_{st}$$ 
		for some $1\le s< t \le 4$.	Since (a) implies $m=0$, we immediately obtain that (a) holds if and only if $n=0$ and  $\x=\y$.  In this case, $\sigma(\x)=\sigma(\y)$.
		 Assuming (b),
		we have $\y=\x+\w+(1-m)\c$. 
		Since $0\le \y \le \vdelta$, we obtain these two inequalities 
		$$ 0\le \sum_{i=1}^{4}(\lambda_i-1)\x_i +(2-m)\c \ \text{ and } \ \sum_{i=1}^{4}(\lambda_i+1)\x_i \le (m+2)\c.$$ 
		Observe that the first inequality is violated for $m\ge 2$ and the second inequality is violated for $m\le 1$. This implies $m=2$ and thus we obtain that (b) holds if and only if $n=4$ and  $\x=\y+\s$. 
		 In this case, $\sigma(\x)=\sigma(\y)+4$.  
		Assume that (c) holds for some $1\le s<t\le 4$, we have $\y=\x-\x_s-\x_t+(1-m)\c$. By a similar argument above, one can check that $m=1$ and thus  we obtain that (c) holds if and only if $n=2$ and
		 $\x=\y+\x_s+\x_t$.  In this case, $\sigma(\x)=\sigma(\y)+2$.	 
	
	\emph{Case $2$}: In this case $n$ is odd and  write $n:=2m+1$. We continue to  deal with the expression $H:=\underline{\Hom}(E,E(\y-\x+m\c)[1])$. By Remark \ref{comp. 2-coAus}, 
	$H\neq 0$  if and only if $\OO(\y-\x+m\c)$ is a direct summand of the injective hull $\mathfrak{I}(E)$.
	By Theorem \ref{inj.hull}, 
	we need to check whether 
	$\y-\x+m\c$ equals to one of the following cases:
	$$\text{($\alpha$)} \ \y-\x+m\c=\w+\x_i \ \ \text{or} \ \ \text{($\beta$)} \ \y-\x+m\c=-\x_i$$ 		
	for some $1\le i \le 4$.  Assume that ($\alpha$) holds for some $i$,  we have $\y=\x+\x_i+\w-m\c$. Since $0\le \y \le \vdelta$, we obtain the two inequalities 
	$$0\le \lambda_i\x_i+\sum_{j\neq i}(\lambda_j-1)\x_j+(1-m)\c \ \text{ and } \ (\lambda_i+2)\x_i+\sum_{j\neq i}(\lambda_j+1)\x_j \le (m+3)\c.$$
	Observe that the first inequality is violated for $m\ge 2$ and the second inequality is violated for $m\le 0$. This implies $m=1$ and thus we obtain that ($\alpha$) holds if and only if $n=3$ and $\x=\y+\s-\x_i$.  In this case, $\sigma(\x)=\sigma(\y)+3$. Assume that ($\beta$) holds for some $i$, we have $\y=\x-\x_i-m\c$.  By a similar argument above, one can check that $m=0$ and thus  we obtain that (c) holds if and only if $n=1$ and
	$\x=\y+\x_i$.  In this case, $\sigma(\x)=\sigma(\y)+1$. Hence we have the assertion.	
	\end{proof}

	\begin{lemma} \label{gen. pre-con} 	We have $E_L(\y)\in \thick T$ for any $\y \in \L$.
\end{lemma}
\begin{proof} Without loss of generality, we let $L:=\OO$. Write $\y\in \L$ in normal form $\y=\sum_{i=1}^{4}\ell_i\x_i+\ell\c$ with $0\le \ell_i < p_i$ and $\ell\in \Z$. Note that we have $E[2]=E(\c)$ by Proposition \ref{susp.ext.}. Since
	$\thick T$ is closed under the suspension [1], 
	 we can assume $\ell=0$ and then $\y=\sum_{i=1}^{4}\ell_i\x_i $. 
	Moreover, by our assumption we have that $E(\y)\in\thick T$ except the case  $\ell_i=p_i-1$  for at least one $1\le i \le 4$. Concerning the number of elements of the set $$\mathcal{S}:=\{  1\le i\le 4\ \mid \ell_i=p_i-1 \}.$$ 	
	We proceed by induction on $|S|$. It is trivial for $|S|=0$. Next we consider the induction step. Let  $|S|=n\ge 1$ and without loss of generality, we can assume $\ell_i=p_i-1$ for $1\le i \le n$.
	We have a sequence of triangles $\eta_{a}$ for $1\le a \le p_1-2$
	\begin{equation} \label{tra. T_12}
		\eta_{a}: E \langle (a-1)\x_1 \rangle \to E \langle a\x_1 \rangle \to E (a\x_1)\to E \langle (a-1)\x_1 \rangle[1]
	\end{equation}		
	by Corollary \ref{triangle2}. Let $\z:=(p_2-1)\x_2+\dots+(p_n-1)\x_n$. We
	consider the  triangles $\eta_{a}(\z)$ obtained from (\ref{tra. T_12}) by twisting with $\z$:
	\begin{equation*} 
		E \langle (a-1)\x_1 \rangle(\z) \to E \langle a\x_1 \rangle(\z) \to E (a\x_1+\z)\to.
	\end{equation*}
		By the induction hypothesis, we have $E(m\x_1+\z)\in \thick T$ for any $0\le m \le p_1-2$. Moreover, 
	by recursively analysing on the triangles $\{\eta_{a}(\z) \mid 1\le a \le p_1-2\}$ from $a=1$ to $p_1-2$,  
	we obtain  $E\langle (p_1-2)\x_1 \rangle(\z)\in \thick T$. By  Corollary \ref{susp. tri. cat.}, we have 
	$$E((p_1-1)\x_1+\z)=E\langle (p_1-2)\x_1 \rangle(\z)[1]\in \thick T,$$
	which concludes the induction step and thus the assertion follows.
\end{proof}

		In order to give a description of the endomorphism algebra $\End(T)^{\rm op}$ in terms of a tensor product, we introduce a class of finite dimensional algebras by modifying the definition of CM-canonical algebras in \cite{HIMO}.
		
		\begin{definition} Let $R$ be an $\L$-graded quadrangle singularity with weight type $(p_1,\dots,p_4)$. Let $\mathbf{q}:=(q_1,\dots,q_4)$ be a quadruple of integers with $2\le q_i\le p_i$ and
			$S:=R/(X_i^{q_i} \mid 1\le i \le 4)$. Denote by $\phi:R\to S$  the canonical morphism, and $I:=[0,\vdelta]$. We define a $k$-algebra
		$$\Lambda(\mathbf{q}):=(\Lambda_{\x,\y}){}_{\x,\y\in I}, \ \  \Lambda_{\x,\y}:=\phi(R_{\x-\y}),$$
		where  the multiplication of $\Lambda(\mathbf{q})$ is given by
		\[(a_{\x,\y})_{\x,\y\in I}\cdot(a'_{\x,\y})_{\x,\y\in I}:=
		(\sum_{\z\in I}a_{\x,\z}\cdot a'_{\z,\y})_{\x,\y\in I}.\]		
	\end{definition}

	\begin{proposition} \label{iso. algs} In the setup above,  there is an isomorphism of k-algebras
		$$\Lambda(\mathbf{q}) \simeq \bigotimes_{1\le i \le 4} k\vec{\AA}_{p_i-1}(q_i).$$ 
	\end{proposition}
	
	\begin{proof}  Let $0\le \x,\y \le \vdelta$ with $\x=\sum_{i=1}^{4}\lambda_i\x_i$ and $\y=\sum_{i=1}^{4}y_i\x_i$ in the normal forms.
		We equip $I:=[0,\vdelta]$ with the structure of a linear order set: 
		$$\x\preceq \y  \ \ \text{if and only if} \ \ (\lambda_1,\dots,\lambda_4)\le_{\rm lex} (y_1,\dots,y_4),$$ where $\le_{\rm lex}$ denotes the usual  lexicographic order, that is, either $\lambda_i=y_i$ for all $1\le i \le 4$, or $\lambda_j=y_j$ for  $j< m$ and  $\lambda_m<y_m$ for some $1\le m \le 4$.  
		Arranging $I$ in this order, it is easy to check that
		\begin{align*}
			\Lambda(\mathbf{q}) &\simeq \bigotimes_{1\le i \le 4}(kI_{p_i-1}+X_iJ_{p_i-1}+\dots+X_i^{q_i-1}J_{p_i-1}^{q_i-1})\\
			& \simeq \bigotimes_{1\le i \le 4} k\vec{\AA}_{p_i-1}(q_i),
		\end{align*}
		where $I_n$  denotes the identity  matrix of size $n$ and $J_n:=\left(\begin{smallmatrix} 0 & I_{n-1} \\ 0 & 0  \end{smallmatrix}\right)$.		
	\end{proof}	 
	
	Note that for the case $\mathbf{q}=(p_1,\dots,p_4)$, the $k$-algebra $\Lambda(\mathbf{q})\simeq \bigotimes_{1\le i \le 4} k\vec{\AA}_{p_i-1}$,
	which reduces to the CM-canonical algebra of the Geigle-Lenzing hypersurface $(R,\L)$ of dimension three defined in \cite{HIMO}.
	
	\begin{proposition} \label{shape theorem}
		The $k$-algebra $\Lambda(\mathbf{q})$ is presented
		by the quiver $Q$ with vertices $Q_0:=[0, \vdelta]$ and arrows
			$$Q_1:=\{\x\xrightarrow{x_i} \x+\x_i \mid 1\le i \le 4 \ \text{and} \ \x, \x+\x_i\in [0, \vdelta]  \}$$
			with the following relations:
			\begin{itemize}
				\item  $x_ix_j=x_jx_i:\x \to \x+\x_i+\x_j$, where $1\le i,j \le 4$ and $0\le \x \le \x+\x_i+\x_j \le \vdelta$,
				\item  $x_i^{q_i}:\x \to \x+q_i\x_i$, where $1\le i \le 4$ and $0\le \x \le \x+q_i\x_i \le \vdelta$.
			\end{itemize}
	\end{proposition}
	\begin{proof} The vertice $\x$ of $Q$  corresponds to the primitive idempotent $e_{\x}$ of $\Lambda:=\Lambda(\mathbf{q})$ and the arrow $x_i$ of $Q$  corresponds to the generator $X_i$ of $S$ for $1\le i \le 4$. Thus there exists a morphism $kQ \to \Lambda$ of $k$-algebras extending by these correspondences, which is surjective. The commutativity relations $X_iX_j=X_jX_i$ are satisfied in $S$. Also the relations $X_i^{q_i}=0$ in $S$  correspond to the relations $x_i^{q_i}=0$ in $kQ$. Thus there is a surjective morphism $kQ/I\to \Lambda$, where $I=(x_ix_j-x_jx_i, x_i^{2}\mid 1\le i,j \le 4).$ Indeed, this is a $k$-algebra isomorphism  since it clearly  induces an isomorphism	$$e_{\x}(kQ/I)e_{\y}\simeq e_{\x}\Lambda e_{\y}= \phi(R_{\x-\y})$$ for any $\x,\y \in Q_0$, as $k$-vector spaces.	Hence we have the assertion.
	\end{proof}

		Now we are ready to complete the proof of Theorem~\rm\ref{main  theorem}.
		
		\begin{proof}[Proof of Theorem~\rm\ref{main theorem}] Without loss of generality, we assume $L:=\OO$. 
			
		\emph{Step $1$}: We show that 	$\underline{\Hom}(T,T[n])=0$ for any $n\neq 0$. Assume for contradiction 
		that	$\underline{\Hom}(E, E(\y-\x)[\sigma(\x)-\sigma(\y)+n])\neq 0$
		for some $0\le \x,\y \le \vdelta$ and $n \neq 0$. 
		By Proposition \ref{hom.rig}, we obtain $n=0$, a contradiction.
		
		\emph{Step $2$}: We show that $\thick T=\underline{\ACM} \, \X$. Note that $E$ is exceptional in $\underline{\ACM} \, \X$ since $\End(E)=k$ implies $\underline{\End}(E)=k$ and by Step (1) $\underline{\Hom}(E,E[n])=0$ holds for any $n\neq 0$.
		Again by Step (1), the $2$-(co)Auslander bundles $E(\x)[-\sigma(\x)]$, $0\le \x \le \vdelta$ can be ordered in such a way that form an exceptional sequence. Thus the smallest triangulated subcategory $\CC$ of $\underline{\ACM} \, \X$ containing $T$ is generated by an exceptional sequence. By \cite[Theorem 3.2]{Bondal:1989} and \cite[Proposition 1.5]{Bondal:Kapranov:1989}, then $\underline{\ACM} \, \X$ is generated by $\CC$ together with $\CC^{\perp}$. Combining Lemma \ref{gen. pre-con} and Corollary \ref{consists of all 2-Aus}, we have $\CC^{\perp}=0$ and thus $\underline{\ACM} \, \X$ is generated by $T$.
		
		\emph{Step $3$}:			
		Concerning the shape of $T$,
		by Proposition \ref{hom.rig}, we have that
		$$\underline{\Hom}(E,E(\y-\x)[\sigma(\x)-\sigma(\y)])\neq 0 \ \Leftrightarrow \ 0\le \x-\y\le \s.$$ 	
		Moreover, in the case, the space has dimension one. Hence we have that the endomorphism algebra $\underline{\End}(T)^{\rm op}\simeq kQ/I$, where  the vertice $\x$ in $Q$ correspond to the $2$-(co)Auslander bundle $E(\x)[-\sigma(\x)]$ for $0\le \x \le \vdelta$, and the
		arrow $x_i:\x \to \x+\x_i$ with $\x,\x+\x_i \in Q$ and $1\le i \le 4$ correspond to the basis of one dimensional space $\Hom(E(\x+\x_i)[-\sigma(\x+\x_i)],E(\x)[-\sigma(\x)])$,  
		and $I=(x_i^{2},x_ix_j-x_jx_i\mid 1\le i,j \le 4).$
		The remaining assertion follows from Propositions \ref{iso. algs} and \ref{shape theorem}. 
		\end{proof}
		
		\begin{remark} For weight type $(2,a,b,c)$ with $a \ge 2$, the $2$-coAuslander bundles are indeed $2$-Auslander bundles by Corollary \ref{susp. tri. cat.}. Hence
			the tilting object $T$ from Theorem \ref{main  theorem} consists only of $2$-Auslander bundles, which gives a positive answer to a higher version of the question stated in \cite[Remark 6.10]{KLM}. 
			
		\end{remark}
		
		\begin{corollary} 	
			Let $\X$ be a GL projective space attached to an $\L$-graded quadrangle singularity of weight type $(p_1,\dots,p_4)$ with $p_i\ge 2$. 
			The endomorphism algebras 
		$\End(T_{\rm cub})^{\rm op}$ in Theorem {\ref{tilting $4$-cuboid}} and $\End(T)^{\rm op}$ in Theorem 
		{\ref{main  theorem}} are derived equivalent.
		\end{corollary}

					\section{The action of $\L$ on $2$-extension bundles}
		In this section, we will investigate the action of the Picard group $\L$ on $2$-extension bundles. As an application, we obtain 
		an explicit formula for the number of $\L$-orbits 
		of all $2$-extension bundles, which gives a positive answer to a higher version of the open question stated in \cite[Remark 9.4]{KLM}.

		\label{sec:The action of 2-extension bundles}
		\begin{proposition} \label{iso. exten. bundle} Let $\x,\y,\z \in \L$ with $0\le \x,\y\le \vdelta$. Write $\x=\sum_{i=1}^4\lambda_i\x_i$. Then we have $E  \langle \x \rangle \simeq E  \langle \y \rangle(\z)$ if and only if one of the following conditions holds.
			\begin{itemize}			
				\item [(a)] $\y=\x$ and $\z=0;$
				\item [(b)] $\y=\vdelta-\x$ and $\z=\x-\w-\c;$
				\item [(c)] $\y=\sum\limits_{i\in I}(p_i-2-\lambda_i)\x_i+\sum\limits_{i\notin I} \lambda_i\x_i$ and $\z=\sum\limits_{i\in I}(\lambda_i+1)\x_i-\c$ for some subset $I \subset \{1,\dots,4\}$ with $|I|=2$.
			\end{itemize}
		\end{proposition}
		
		\begin{proof} By Proposition \ref{susp.ext.}, it is straightforward to check that if $\y$ and $\z$ satisfy one of the statements above, then we have $E  \langle \x \rangle \simeq E  \langle \y \rangle(\z)$.
			
			Conversely,		assume $E  \simeq F$. For simplicity, we write $E=E  \langle \x \rangle$ and $F=E  \langle \y \rangle(\z)$ and  $\y=\sum_{i=1}^{4} y_i\x_i$.
		 This implies $\mathfrak{P}(E ) \simeq  \mathfrak{P}(F)$, where
			\begin{align*}
				\mathfrak{P}(E)&=\OO(\w) \oplus \Big( \bigoplus_{\substack{I \subset \{1,\dots,4\} \\ |I|=2}}\OO(\sum_{i\in I}(1+\lambda_i)\x_i+\w-\c) \Big) \oplus \OO(\x-\c),\\
				\mathfrak{P}(F)&= \OO(\w+\z) \oplus \Big( \bigoplus_{\substack{I \subset \{1,\dots,4\} \\ |I|=2}}\OO(\sum_{i\in I}(1+\y_i)\x_i+\w+\z-\c) \Big) \oplus \OO(\y+\z-\c).
			\end{align*}
			Thus $\OO(\w)$ is a direct summand of $\mathfrak{P}(F)$. We consider the following three cases.		
			
			\emph{Case $1$}: In this case $\OO(\w)=\OO(\w+\z)$, it follows that $\z=0$. We consider the direct summand $\OO(\x-\c)$ of $\mathfrak{P}(E)$. Assume that $\OO(\x-\c)=\OO(\sum_{i\in I}(1+y_i)\x_i+\w-\c)$ holds for some $I\subset\{1,\dots,4\}$ with $|I|=2$. Then we have $$\sum_{i\in I}(y_i-\lambda_i)\x_i-\sum_{i\not\in I}(1+\lambda_i)\x_i+\c=0,$$ which implies a contradiction. Thus this only happens $\OO(\x-\c)=\OO(\y-\c)$, which implies $\y=\x$. 
			Thus we have the statement (a).			
			
			\emph{Case $2$}: In this case $\OO(\w)=\OO(\y+\z-\c)$,  it follows that $\z=\c+\w-\y$. Assume that there exists a subset $I\subset\{1,\dots,4\}$ with $|I|=2$ such that $\OO(\x-\c)=\OO(\sum_{i\in I}(1+y_i)\x_i+\w+\z-\c)$. Then $\x-\c=\sum_{i\in I}(1+y_i)\x_i+2\w-\y$ yields $$\sum_{i\in I}(1+\lambda_i)\x_i+\sum_{i\not\in I} (\lambda_i+y_i+2)\x_i=3\c.$$ 
			Because the coefficient of $\c$ in the normal form of the left side has at most $2$, this implies a contradiction. Thus this only happens $\OO(\x-\c)=\OO(\w+\z)$, which implies $\z=\x-\w-\c$. Hence we have $\y=\vdelta-\x$ 
			and obtain  the statement (b).

			\emph{Case $3$}: In this case $\OO(\w)=\OO(\sum_{i\in I}(1+\y_i)\x_i+\w+\z-\c)$ for some subset $I\subset\{1,\dots,4\}$ with $|I|=2$, it follows that $\z=\c-\sum_{i\in I}(1+y_i)\x_i$. We consider the direct summand $\OO(\w+\z)$ of $\mathfrak{P}(F)$. Assume that $\OO(\w+\z)=\OO(\x-\c)$ holds. Then we obtain $\x+\sum_{i\in I}(1+y_i)\x_i-2\c-\w=0,$ that is, $$\sum_{i\in I}(\lambda_i+y_i+2)\x_i+\sum_{i\notin I}(1+\lambda_i)\x_i=3\c.$$ Because the coefficient of $\c$ in the normal form of the left side has at most $2$, this implies a contradiction. 
			Then we may assume that there exists
			some $J\subset\{1,\dots,4\}$ with $|J|=2$ such that $\OO(\w+\z)=\OO(\sum_{j\in J}(1+\lambda_j)\x_j+\w-\c)$, that is, $$\sum_{i\in I}(1+y_i)\x_i+\sum_{j\in J}(1+\lambda_j)\x_j-2\c=0.$$  By counting the intersection number of $I$ and $J$,  the equality holds if and only if $I=J$ and $y_i=p_i-2-\lambda_i$ for any $i\in I$. Thus we have $\z=\sum_{i\in I}(1+\lambda_i)\x_i-\c$. 			
			Next we consider the direct summand $\OO(\y+\z-\c)$ of $\mathfrak{P}(F)$. Assume that   $\OO(\y+\z-\c)=\OO(\x-\c)$ holds. Then we have $\x-\y-\z=0$, that is, $$\sum_{i\not\in I}(\lambda_i-y_i)\x_i-\sum_{i\in I}(1+y_i)\x_i+\c=0,$$ a contradiction. Thus we may assume that there exists some subset $U\subset\{1,\dots,4\}$ with $|U|=2$ satisfying $\OO(\y+\z-\c)=\OO(\sum_{i\in U}(1+\lambda_i)\x_i+\w-\c)$, that is, $\y+\z-\sum_{i\in U}(1+\lambda_i)\x_i-\w=0$. This implies
			\begin{align*}
				0&=\sum_{i\in I}(p_i-1)\x_i+\sum_{i\not\in I}y_i\x_i-\sum_{i\in U}(1+\lambda_i)\x_i-\c-\w  \\
				&=\sum_{i\not\in I}(1+y_i)\x_i-\sum_{i\in U}(1+\lambda_i)\x_i.
			\end{align*}
			By comparing the coefficients of $\x_i$, the equality holds if and only if $U=I^c$ and $y_i=\lambda_i$ for any $i\in U$. Therefore, we finally obtain that $$\y=\sum\limits_{i\in I}(p_i-2-\lambda_i)\x_i+\sum\limits_{i\notin I} \lambda_i\x_i.$$
			Hence	$E  \langle \x \rangle \simeq E  \langle \y \rangle(\z)$ implies that one of the three statements holds.					
		\end{proof}
		
		By the argument in the above proof, we indicate that the projective cover and injective hull are both invariants for $2$-extension bundles.
		
		\begin{corollary}\label{projective cover determines extension bundle}
			Let $E$ and $F$ be two $2$-extension bundles. Then the following conditions are equivalent.
			\begin{itemize}
				\item [(a)] $E\simeq F$.		
				\item [(b)] $\mathfrak{P}(E)\simeq \mathfrak{P}(F)$.		
				\item [(c)] $\mathfrak{I}(E)\simeq \mathfrak{I}(F)$.
			\end{itemize}
		\end{corollary}
		
		The following result reveals when a $2$-extension bundle is a $2$-Auslander bundle.
		
		\begin{corollary} \label{ext.is Aus} Let $E\langle \x \rangle$ be a $2$-extension bundle for $0\le \x \le \vdelta$. Then $E\langle \x \rangle$ be a $2$-Auslander bundle if and only if 
			$\x=\sum_{i\in I}(p_i-2)\x_i$ for some even subset $I \subset \{1,\dots,4\}$.			
		\end{corollary}
		
		Denote by $\mathcal{V}$ the set of isoclasses of $2$-extension bundles in $\ACM \X$. As an application of Proposition \ref{iso. exten. bundle}, we obtain an explicit formula of the orbit set $\mathcal{V}/\L$.
		
		\begin{theorem} \label{exact number}
			Let  $J=\{1\le i \le 4 \mid p_i \ \text{is even}\}$. Then we have 
			$$|\mathcal{V}/\L|=\frac{1}{8}\sum_{\substack{I \subset J \\ |I|  \text{ even}}}\prod_{  i\in  I^{c}} (p_i-1),$$
			where $I^c:=\{1,\dots,4\}\setminus I$.
		\end{theorem}
		\begin{proof} Each $2$-extension bundle $E$ has the form $E=E_L \langle\x \rangle$ for some line bundle $L$ and $0\le \x \le \vdelta$. By Observation \ref{shift preserve extension bundle}, $E_L \langle\x \rangle$ lies in  the same $\L$-orbit with $E \langle\x \rangle$. We denote by $\mathcal{S}:=\{E\langle\vec{x}\rangle \mid 0\le \x \le \vdelta\}.$ For any even subset $I\subset \{1,\dots,4\}$, we define a map $$\sigma_I: \mathcal{S}\mapsto\mathcal{S}, \quad
			E\langle\sum\limits_{i=1}^{4}\lambda_i\vec{x}_i\rangle \mapsto E \langle \sum\limits_{i\in I}(p_i-2-\lambda_i)\x_i+\sum\limits_{i\notin I} \lambda_i\x_i \rangle.$$
			It is straightforward to verify that for any even subsets $I,J \subset \{1,\dots,4\}$, we have $\sigma_I\sigma_J=\sigma_J\sigma_I$ and $\sigma_I^2=\sigma_{\emptyset}=1_{\mathcal{S}}$, where $1_{\mathcal{S}}$ is the identity map on $\mathcal{S}$. Then the group $G$ consisting all the elements  $\sigma_I$ has order $8$, which is isomorphism to
			the group $\Z_2 \times\Z_2\times\Z_2$. Consider the action of $G$  on the set $\mathcal{S}$. By proposition \ref{iso. exten. bundle}, any two $2$-extension bundles from $S$ are in the same $\L$-orbit if and only if they are in the same $G$-orbit. By Burnside Formula, we have
			$$|\mathcal{V}/\L|=\frac{1}{|G|}\sum\limits_{\sigma\in G}|\mathcal{S}^{\sigma}|,$$ where $|X|$ denotes the order of the set $X$ and $\mathcal{S}^{\sigma}=\{s\in \mathcal{S}\mid \sigma(s)=s\}$ is the subset of fixed points. By Proposition \ref{iso. exten. bundle}, for any even subset $I\subset \{1,\dots,4\}$, 
			$$\sigma_I(E\langle\sum\limits_{i=1}^{4}\lambda_i\vec{x}_i\rangle)=E\langle\sum\limits_{i=1}^{4}\lambda_i\vec{x}_i\rangle$$ holds if and only if $p_i-2-\lambda_i=\lambda_i$ for any $i \in I$, 
			that is, $p_i$ is even and $\lambda_i=\frac{p_i-2}{2}$. In this case, $|\mathcal{S}^{\sigma_I}|=\prod_{  i\in  I^{c}} (p_i-1).$ Thus we have the assertion.
		\end{proof}
		
		Recall that the action of $\L$ on $\mathcal{V}$ is \emph{transitive} if there exists a unique orbit, that is, $|\mathcal{V}/\L|=1$. As an immediate consequence, we have the following result.
		
		\begin{corollary} The action of $\L$ on the set of all $2$-extension bundles is transitive if and only if $\X$ has weight quadruple $(2,2,2,2)$, $(2,2,2,3)$, $(2,2,3,3)$.  			
		\end{corollary}
		
		\appendix
		\section{The Proof of  exceptionality of $\underline{C}$} \label{excep C}
		The aim of this section is to give a complete proof of exceptionality of $\underline{C}$ in Proposition  {\rm \ref{triangle}}. Recall that there are exact sequences \begin{align*}
			\xi:&\quad 0 \to E_L \langle \x \rangle \xrightarrow{v_i}  E_L \langle \x+\x_i \rangle \to U_i \to 0,\\	
			\zeta:&\quad	0 \to  L_i[\x] \oplus \bigoplus_{j\neq i} L_j\langle \x \rangle \to \underline{C} \to U_i \to 0,
		\end{align*}
		where $0 \le \x \le \x+\x_i \le \vdelta$ and $v_i:=x_i^{\ast}$ given in Proposition {\ref{ext-pullback}}.

		\begin{proposition} \label{pre excep of C}  $U_i$ is  an exceptional object in $\coh \X$.			
		\end{proposition}
		\begin{proof} 
			We put $\z:=\x+\x_i$ and let $\eta_{\x}:0 \to L(\w) \to E_L\langle \x \rangle \to L\langle \x \rangle \xrightarrow{\gamma} L(\x) \to 0$ be the defining sequence for the $2$-extension bundle $E_L\langle \x \rangle$. The proof is divided into the following three steps.
			
			\emph{Step $1$}: We show that $\Hom(E_L\langle \z \rangle,U_i)=k$ and $\Ext^{j}(E_L\langle \z \rangle,U_i)=0$ for $j=1,2$.
			First we claim that $\Ext^{j}(E_L\langle \z \rangle,E_L\langle \x \rangle)=0$ holds for any $j\ge 0$.
			 Factor $\eta_{\x}$ into these two short exact sequences	$$\eta_1: 0\to L(\w)\to E_L\langle \x \rangle \to K_{\x} \to 0 \ \ \text{and} \ \ \eta_2: 0 \to K_{\x} \to L\langle \x \rangle \xrightarrow{\gamma} L(\x) \to 0.$$
			By applying $\Hom(E_L\langle \z \rangle,-)$ to  $\eta_2$, we obtain an exact sequence 
			\begin{align*}
				0 &\to (E_L\langle \z \rangle,K_{\x}) \to (E_L\langle \z \rangle,L \langle \x \rangle) \xrightarrow{\gamma^{\ast}} (E_L\langle \z \rangle,L(\x)) \\ 
				&\to {}^{1}(E_L\langle \z \rangle,K_{\x})  \to{}^{1}(E_L\langle \z \rangle,L \langle \x \rangle)  \to {}^{1}(E_L\langle \z \rangle,L(\x))  \\ 
				&\to{}^{2}(E_L\langle \z \rangle,K_{\x}) \to{}^{2}(E_L\langle \z \rangle,L \langle \x \rangle) \to{}^{2}(E_L\langle \z \rangle,L(\x)) \to 0. \nonumber
			\end{align*}
			 Clearly, $\Ext^{1}(E_L\langle \z \rangle,L(\x))=0$ and $\Ext^{1}(E_L\langle \z \rangle,L\langle \x \rangle)=0$ hold by (\ref{ACM bundle eq.}). 
			Invoking Lemma \ref{applying hom}, by applying $\Hom(-,L(\x))$ and $\Hom(-,L\langle \x \rangle)$ 
			to the defining sequence of $E_L\langle \z \rangle$, respectively, it is straightforward to verify that  
			\begin{align*}
				\Hom(E_L\langle \z \rangle,L(\x))=k \ \ \text{and} \ \ \Ext^{2}(E_L\langle \z \rangle,L(\x))=0, \\
				\Hom(E_L\langle \z \rangle,L\langle \x \rangle)=k \ \ \text{and} \ \ \Ext^{2}(E_L\langle \z \rangle,L\langle \x \rangle)=0.
			\end{align*}
			The morphism $\gamma^{\ast}$ between one dimensional spaces is nonzero, implying that $\gamma^{\ast}$ is an isomorphism. Thus we have $\Ext^{j}(E_L\langle \z \rangle, K_{\x})=0$ for any $j \ge 0$. Then by applying $\Hom(E_L\langle \z \rangle,-)$ to $\eta_{1}$, we obtain an exact sequence
			\begin{align*}
			0 &\to (E_L\langle \z \rangle,L(\w)) \to (E_L\langle \z \rangle,E_L\langle \x \rangle) \to (E_L\langle \z \rangle,K_{\x}) \\ 
				&\to {}^{1}(E_L\langle \z \rangle,L(\w))  \to{}^{1}(E_L\langle \z \rangle,E_L\langle \x \rangle)  \to {}^{1}(E_L\langle \z \rangle,K_{\x})  \\ 
				&\to{}^{2}(E_L\langle \z \rangle,L(\w)) \to{}^{2}(E_L\langle \z \rangle,E_L\langle \x \rangle) \to{}^{2}(E_L\langle \z \rangle,K_{\x}) \to 0. \nonumber
			\end{align*}
			By Auslander-Reiten-Serre duality and Proposition \ref{homological properties of $E$ and $K$ 1}, for any $j\ge 0$, we have $$\Ext^{j}(E_L\langle \z \rangle,L(\w))\simeq D\Ext^{2-j}(L,E_L\langle \z \rangle)=0.$$
			Thus we have $\Ext^{j}(E_L\langle \z \rangle,E_L\langle \x \rangle)=0$ for any $j\ge 0$ and the claim follows.
			By applying $\Hom(E_L\langle \z \rangle,-)$ to  $\xi$, by the exceptionality of  $E_L\langle \z \rangle$ in $\coh \X$, we obtain  $\Hom(E_L\langle \z \rangle,U_i)=k$ and $\Ext^{j}(E_L\langle \z \rangle,U_i)=0$  for $j=1,2$.

			\emph{Step $2$}: We claim that $\Ext^{j}(E_L\langle \x \rangle,U_i)=0$ for $j=0,1$. Applying $\Hom(E_L\langle \x \rangle,-)$ to the sequence $\xi$, we obtain an exact sequence
			\begin{align*}
					0 &\to (E_L\langle \x \rangle,E_L\langle \x \rangle) \to (E_L\langle \x \rangle,E_L\langle \z \rangle) \to (E_L\langle \z \rangle,U_i) \\ 
				&\to {}^{1}(E_L\langle \x \rangle,E_L\langle \x \rangle)  \to{}^{1}(E_L\langle \x \rangle,E_L\langle \z \rangle)  \to {}^{1}(E_L\langle \z \rangle,U_i)  \\ 
				&\to{}^{2}(E_L\langle \x \rangle,E_L\langle \x \rangle) \to{}^{2}(E_L\langle \x \rangle,E_L\langle \z \rangle) \to{}^{2}(E_L\langle \z \rangle,U_i) \to 0. \nonumber
			\end{align*}
			Combining Proposition \ref{homological properties of $E$ and $K$ 2}, Lemma \ref{ext1(E,E)=0} and Proposition \ref{ext-pullback},  the claim follows.
			
			\emph{Step $3$}: 
			Applying $\Hom(-,U_i)$ to the sequence $\xi$, we have an exact sequence
			\begin{align*}
				0 &\to (U_i,U_i) \to (E_L\langle \z \rangle,U_i) \to (E_L\langle \x \rangle,U_i) \\ 
				&\to {}^{1}(U_i,U_i)  \to{}^{1}(E_L\langle \z \rangle,U_i)  \to {}^{1}(E_L\langle \x \rangle,U_i)  \\ 
				&\to{}^{2}(U_i,U_i) \to{}^{2}(E_L\langle \z \rangle,U_i) \to{}^{2}(E_L\langle \x \rangle,U_i) \to 0. \nonumber
			\end{align*}
			Putting things together, we have shown  the exceptionality of $U_i$ in $\coh \X$.		
		\end{proof}
		
		Now we are ready to prove the exceptionality of $\underline{C}$ in Proposition  {\rm \ref{triangle}}.
		
		\begin{proposition} \label{excep of C}  $\underline{C}$  is an exceptional object in $\coh \X$.			
		\end{proposition}
		\begin{proof} We write $\z:=\x+\x_i$ and $V:=L_i[\x] \oplus \bigoplus_{j\neq i} L_j\langle \x \rangle$. Applying $\Hom(V,-)$ to the sequence $\xi$, we obtain an exact sequence
			\begin{align*}
				0 &\to (V, E_L \langle \x \rangle) \to (V,E_L \langle \z \rangle) \to (V,U_i) \\ 
				&\to {}^{1}(V, E_L \langle \x \rangle)  \to{}^{1}(V,E_L \langle \z \rangle)  \to {}^{1}(V,U_i)  \\ 
				&\to{}^{2}(V, E_L \langle \x \rangle) \to{}^{2}(V,E_L \langle \z \rangle) \to{}^{2}(V,U_i) \to 0. \nonumber
			\end{align*}
		It is straightforward to verify that  $\Ext^{j}(V,E_L \langle \x \rangle)=0$ and $\Ext^{j}(V,E_L \langle \z \rangle)=0$ hold for any $j\ge 0$ by Lemma \ref{applying hom} and thus we have $\Ext^{i}(V,U_i)=0$ for any $i \ge 0$. Then applying $\Hom(-,U_i)$ to the sequence $\zeta$, and combining the exceptionality of $U_i$ shown in Proposition \ref{pre excep of C}, we obtain $$\Hom(\underline{C},U_i)=k \ \ \text{and} \ \ \Ext^{j}(\underline{C},U_i)=0 \ \text{for} \ j=1,2.$$ 
		 Similarly, by applying $\Hom(-,V)$ to $\xi$, one can check that 
		$$\Hom(U_i,V)=0, \ \ \Ext^{1}(U_i,V)=k^{4} \ \ \text{and} \ \,  \Ext^{2}(U_i,V)=0.$$		
		Now we apply $\Hom(-,V)$ to $\zeta$ and obtain an exact sequence 
			\begin{align*}
				0 &\to (U_i,V) \to (\underline{C},V) \to (V,V) \\ 
				&\to {}^{1}(U_i,V)  \to{}^{1}(\underline{C},V)  \to {}^{1}(V,V)  \\ 
				&\to{}^{2}(U_i,V) \to{}^{2}(\underline{C},V) \to{}^{2}(V,V) \to 0. \nonumber
			\end{align*}
			Note that $\underline{C} \in \ACM \X$, then we have $\Ext^{1}(\underline{C},V)=0$.
		It is straightforward to verify that $\Hom(V,V)=k^{4}$	and $\Ext^{j}(V,V)=0$ for $j=1,2$, implying that $\Ext^{j}(\underline{C},V)=0$ for all $j\ge 0$.
		Finally we apply $\Hom(\underline{C},-)$ to  $\zeta$ and obtain an exact sequence
			\begin{align*}
				0 &\to (\underline{C},V) \to (\underline{C},\underline{C}) \to (\underline{C},U_i) \\ 
				&\to {}^{1}(\underline{C},V)  \to{}^{1}(\underline{C},\underline{C})  \to {}^{1}(\underline{C},U_i)  \\ 
				&\to{}^{2}(\underline{C},V) \to{}^{2}(\underline{C},\underline{C}) \to{}^{2}(\underline{C},U_i) \to 0. \nonumber
			\end{align*}
		 Putting things together, we  have shown  the exceptionality of $\underline{C}$ in $\coh \X$.
		\end{proof}

		\noindent {\bf Acknowledgements.} Jianmin Chen and Weikang Weng were partially supported by the National Natural Science Foundation of China (Nos. 12371040 and 12131018). Shiquan Ruan was partially supported by the Natural Science Foundation of Xiamen (No. 3502Z20227184), Fujian Provincial Natural Science Foundation of China (Nos. 2024J010006 and 2022J01034), the National Natural Science Foundation of China (Nos. 12271448), and the Fundamental Research Funds for Central Universities of China (No. 20720220043).

			\vskip 5pt
		\noindent {\scriptsize   \noindent Jianmin Chen, Shiquan Ruan and Weikang Weng\\
			School of Mathematical Sciences, \\
			Xiamen University, Xiamen, 361005, Fujian, PR China.\\
			E-mails: chenjianmin@xmu.edu.cn, sqruan@xmu.edu.cn,
			wkweng@stu.xmu.edu.cn\\ }
		\vskip 3pt
		
	\end{document}